\documentclass[11]{article}

\usepackage{amssymb, amsmath, verbatim, amsthm,url, multirow,fullpage,mathtools, appendix, cancel, graphicx, subcaption, verbatim, multirow}
\usepackage{longtable, rotating,makecell,array}
\usepackage[aligntableaux=top]{ytableau}

\usepackage{soul}
\usepackage[colorinlistoftodos,textsize=footnotesize]{todonotes}
\newcommand{\hlfix}[2]{\texthl{#1}\todo{#2}}

\newcommand{\sanote}{\todo[color=violet!30]}
\newcommand{\fknote}{\todo[color=orange!10]}

\setstcolor{red}

\newcommand{\Z}{\mathbb{Z}}

\newcommand{\R}{\mathbb{R}}

\newcommand{\namedfilename}{namedminus2-2p}
\newcommand{\fullfilename}{fullminus1-2p}

\def\ba #1\ea{\begin{align} #1 \end{align}}
\def\bas #1\eas{\begin{align*} #1 \end{align*}}
\def\bml #1\eml{\begin{multline} #1 \end{multline}}
\def\bmls #1\emls{\begin{multline*} #1 \end{multline*}}

\newcommand{\varep}{\varepsilon}

\newcommand{\vrho}{\vec{\rho}}

\newcommand{\namedtwo}{G_{\textrm{named}}^{(-2)}}
\newcommand{\fullone}{G_{\textrm{full}}^{(-1)}}
\newcommand{\rB}{\mathcal{B}}

\newtheorem{thm}{Theorem}[section]

\newtheorem{lem}[thm]{Lemma}

\theoremstyle{remark}

\theoremstyle{definition}
\newtheorem{dfn}[thm]{Definition}
\newtheorem{rmk}[thm]{Remark}

\title{A Geometric Chung Lu model and the Drosophila Medulla connectome}

\author{Susama Agarwala\footnote{Johns Hopkins University, Applied Physics Lab; Laurel, MD 20723. susama.agarwala@jhuapl.edu}\\
        Franklin Kenter\footnote{Mathematics Department, United States Naval Academy; Annapolis, MD 21402. kenter@usna.edu} \\
}

\date{}

\begin{document}
\maketitle


\begin{abstract}
    Many real world graphs have edges correlated to the distance between them, but, in an inhomogeneous manner. While the Chung-Lu model and the geometric random graph models both are elegant in their simplicity, they are insufficient to capture the complexity of these networks. In this paper, we develop a generalized geometric random graph model that preserves many graph theoretic aspects of these real world networks. We test the validity of this model on a graphical representation of the Drosophila Medulla connectome.
\end{abstract}

\section{Introduction}

In many networks, the presence of an edge between two nodes strongly correlates with the spatial distance between them; this includes infrastructure networks,  transportation networks \cite{spatialsurvey} and connectomes \cite{marmoset,humanspatialconnectome}. Traditionally, scale free networks are used to model these systems, but there is increasing evidence that this is too simplistic  a paradigm \cite{powerlawscarce}. For instance, in large scale connectome networks of regions of the human brain, it has been shown that the empirical probability that an edge exists between two nodes is a function of distance which decays faster than any inverse power decay\cite{humanspatialconnectome}.

To address is shortcoming, we build a composite model for  that is a combination of  two prototypical random graph models: the {\it random geometric graph model} and the {\it Chung-Lu} model. We validate this composite model on the connectome formed by tracing the neurons in a small portion of the Drosophila Medulla\cite{neurodata}. In doing so, we find that we are able to generate a family of graphs that preserve many of the network properties observed in the original connectome in a way that could not have been possible just using one or the other of the two component models. 

There is a lot of evidence in favor of spatial and geometric models for the brain in meso- and macro-scale connectome work \cite{spatialagainst,connectometutorial, spatialfor} that transcends species. However, continuing this work on a micro, i.e. a cellular level has been stymied by the difficulty and expense of obtaining the data. To address this issue of sparse examples of mapped connectomes, we propose a means of constructing a composite model that generates synthetic graphs which share many of the graph characteristics of the graph defined by the Drosophilla medulla \cite{naturenetwork}.



In the random geometric graph model, vertices are randomly embedded into some space with a metric (usually Euclidean); after which, two vertices $i, j$ are adjacent with a probability {\it connection function} based upon the spatial distance between them $d_{ij}$. The simplest case is when an edge is present whenever $d_{ij} < R$ for some threshold value $R$ \cite{hardpenrose}. Though more complicated connection functions can be used depending on the application \cite{softpenrose, generalconnection}. One of the main shortcomings of random geometric graph models is the homogeneity of the vertices which stands in contrast to most real-world networks \cite{estradahetero}. The Chung-Lu model \cite{chunglu} is a non-geometric heterogeneous model where each vertex is randomly assigned a weight $w_i$ from a predetermined distribution. Then, the probability that an edge exists between two vertices is a function of the weights $w_i$ and $w_j$. Under mild conditions, the expected degrees of the vertices are precisely $w_i$, allowing for heterogeneity \cite{chunglu}. This model allows straightforward analysis of a general degree distribution, especially those corresponding to ``power-law'' or ``scale-free'' networks \cite{chunglu}. However, it has been observed that the Chung-Lu model tends to underestimate triangle counts in real-world graphs, and miss estimates degrees for real world networks that do not satisfy the Chung-Lu constraint
\cite{notchunglu}. The connectome studied in this paper (and connectomes in general) has heterogeneous vertices, as well as a geometrically defined connection function, making it an ideal candidate for a composite of these two models.

There are several previously-used methods for combining geometric graphs with the Chung-Lu model. One method is to use a hyperbolic embedding where the connection function $f(d_{ij},w_i,w_j)$ is a function of the hyperbolic distance as well as the two weights. The advantage of using hyperbolic space is that the resulting clustering is similar to that found in social networks \cite{hyperbolicsolcial}. In this model, as the number of vertices $\to \infty$, the limit of the hyperbolic model is the classical Chung-Lu model  \cite{hyperchunglu}. The hyperbolic model can be generalized using a toroidal embedding in different dimensions \cite{bringmann2019geometric}. However, in all cases, the connection function is proportional to  $d_{i,j}^{-\beta}$ for some $\beta$ (i.e. an inverse power law) \cite{bringmann2019geometric}. This stands in contrast to the rate of decay for real-world graphs mentioned previously.  

To address these concerns, in Section \ref{sec:genericmodel}, we introduce a simple yet robust heterogeneous geometric random graph model that is able to consider both the spatial position of the vertices as well as a randomly assigned weight. In the hyperbolic model above, the probability that two nodes are connected, given their distance is prescribed by a simple power law (namely $d_{i,j}^{-\beta}$). In our model, we generalize this to allow for any function that limits to the edge density of the graph as the distance goes to $\infty$ and to the self loop number as the distance goes to $0$. 

Then, in Section \ref{sec:specific}, we adapt this model to consider the boundary effects when the spatial coordinates are restricted to a (usually non-convex) region. Together, this allows us to build synthetic connectomes resulting in similar local and global topological properties as the original as well as maintaining the original spatial characteristics and spectral characteristics.  In Section \ref{sec:connectomesimulation}, we apply our model to part of the connectome of the {\it Drosophila medulla} given in \cite{naturenetwork} wherein synaptic connections are detailed in their exact spatial coordinates in three dimensions. In Section \ref{sec:simulation}, we use a variety of network measures, we show that synthetic connectomes generated under our model are more similar to those created under a pure Chung-Lu model. Finally, Section \ref{sec:othermodels}, we compare the accuracy of our model to other existing models.

\section{Chung-Lu model and existing generalizations \label{sec:litreview}}

For our purposes, the graphs we consider, $G = (V, E)$, are finite and undirected; we do not allow for multiedges but we allow for graphs to have self-loops. Throughout, we fix the number of vertices in the graph, $n = |V|$. We write $i \sim j$ to denote that $i$ is adjacent to $j$. A self-loop occurs whenever $i\sim i$. The degree of a vertex is the number edges incident to it, denoted $\deg_i$; for our purposes, self-loops only contribute 1 toward the degree of a vertex. The {\it edge-density} of a graph is $\varep = \frac{2|E|}{(n+1)n}$.

We work with graphs that are embedded in some geometric space, therefore each vertex has an associated position. We denote the Euclidean distance between two vertices, $i$ and $j$, as $d_{i,j}$. Note that this is the distance between two vertices relative to the geometry the graph is embedded in, not the graph-theoretic distance. 

\subsection{The Original Chung-Lu Model \label{sec:CLorig}}

The Chung-Lu model \cite{chunglu} is a heterogeneous random graph model whose parameter is a single vector of expected degrees: $\ \vec w$ of length $n$. Given this vector, a graph can be generated where the probability that two vertices (not necessarily distinct) are adjacent is 
\[ P(i \sim j | w_i, w_j) = \frac{w_i w_j}{\sum_k w_k}\;, \]
and the event $i \sim j$ is independent from any other pair.
This probability is valid for all pairs $i, j$ if
\ba  (\max_\ell w_\ell)^2 < \sum_k w_k \label{eq:CLcondition}\ea
We refer to equation \ref{eq:CLcondition} as the ``Chung-Lu condition''.

\subsection{Previous Geometric Chung-Lu Models\label{sec:CLgeomother}}

Bringmann, Keusch and Lengler  introduced a Geometric Chung-Lu Model, called Geometric Inhomogeneous Random Graphs (GIRG), \cite{bringmann2019geometric} which generalizes a previous geometrized Chung-Lu model, the hyperbolic model \cite{hyperchunglu}. In this model, $n$ vertices are each assigned a random point chosen uniformly in the $d$-dimensional unit torus, and each vertex is assigned a random weight from $\ \vec w$ of length $n$. The connection function (i.e., the probability of an edge $i\sim j$) is given by
\[ P(i \sim j | d_{ij}, w_i, w_j ) = \min \left \{  \frac{1}{d_{i,j}^{\alpha d}} \cdot \left( \frac{w_i w_j}{W} \right)^{d}, 1 \right \} \]
where $d_{i,j}$ is taken to be the $\infty$-norm on the torus and $W = \sum_k w_k$. As before, this probability is independent for each pair.  Unlike the traditional Chung-Lu model, the vector $\vec w$ is not necessarily the expected degree of each vertex. 
Indeed, under mild conditions for $d$, $\alpha$ and $\vec w$, the probability that two randomly chosen vertices are adjacent given their distance (but without knowing the $w_i$ or $w_j$) is 
\[ P(i\sim j | d_{ij}) = \frac{1}{d_{ij}^{\alpha d}} \cdot \frac{1}{n(n-1)} \cdot W^d  \]

In particular, regardless of $\vec w$, $d$ and $\alpha$, the edge probability decreases according to an inverse power law, $C d_{ij}^{-\beta}$. However, this is not necessarily the case for real-world networks, for instance in meso-scale connectome networks \cite{humanspatialconnectome,connectometutorial,marmoset}.
In Section \ref{sec:othermodels}, we show that this is also not true for the micro scale connectome of the Drosophilla Medulla studied here.

In contrast, in the model presented in the paper, we relax this restriction by allowing for a near arbitrary function with the potential of exponential decay. 

\section{A Generic Geometric Chung-Lu Model \label{sec:geommodel}} 

In this section, develop an generic geometric Chung-Lu Model. Unlike the geometric inhomogeneous random graphs above, our connection function will decay at an exponential rate. We call this a generic model because we relax two fundamental assumptions of previous models to accommodate real world structures. Namely, we do not require the nodes to either be distributed uniformly over a space or have any periodic properties.

We do this in two stages. In section \ref{sec:genericmodel}, we define a mathematically simple model where vertices are embedded into a $d$-dimensional torus under some distribution $\mu$. Most real-world graphs are not periodic in nature and have boundary considerations. Therefore, in section \ref{sec:specific}, we restrict our graph to one that is defined on a subset of the
torus $\rB \subset \R^d/ \Z^d$, with all vertices (and edged between vertices) outside $\rB$ removed, and all edges between vertices inside and outside of $\rB$ removed.
Unlike the traditional Chung-Lu model, { or even the geometric inhomogeneous random graphs model,} the boundary condition implies that the vector of weights $\vec{\rho}$ does not correspond directly with the expected degree of each node. Hence, the boundary conditions requires a careful re-weighting of the vertices to find the appropriate weights.

\subsection{A Geometric Chung-Lu Model without uniformity \label{sec:genericmodel}} 

The following are the inputs for the generic geometric Chung-Lu model on a $d$ dimensional torus. 
\begin{dfn}[Parameters of the Model] \label{dfn:genericinputs} ~\\

\begin{enumerate}
    \item \textbf{The dimension}, $d$.
    \item \textbf{The number of vertices}, $n$.
    \item \textbf{A probability measure.} A measure on $\mathbb{R}^d / \mathbb{Z}^d$, $\mu$, describing the distribution of nodes.
    \item \textbf{One parameter family of connection probabilities.} We define a function that gives a joint probability that two vertices are adjacent and within a certain distance from each other. We call this function \ba F_1(x) = P(i \sim j, d_{i,j} \le x)\label{eq:distandconnect} \ea.  We will place nontrivial constraints on $F_1$ that we describe later in Definition {\ref{dfn:genericassumptions}}.
    \item \textbf{Intensity vector.} A vector of intensities, $\vec{\rho}$, whose length is the number of nodes $n$.
\end{enumerate}
\end{dfn}

As with the Chung-Lu model, the vector $\vec{\rho}$ will, under mild conditions, be the expected degree of the nodes. However, when we introduce boundary considerations, this will no longer be the case. To emphasize this distinction, we refer to the vector as a vector of ``intensities'' as opposed to ``weights.''

To generate a graph under this model, we first place $n$ vertices independently on a $d$-dimensional torus, $\R^d$, according to the distribution $\mu$.

Once the $n$ vertices have been assigned a position on the torus, one randomly associates to each vertex an element of the intensity vector without replacement. (Throughout, we will mildly abuse notation and represent this by writing $P(\rho_i)$ which is $\frac{1}{n}$).

Given these inputs, we next construct the connection function for the model. In order to do so, we need to impose some additional assumptions and constraints. In particular, we need to ensure that the expected edge densities from the Chung-Lu and geometric aspects of the model agree. This, and other necessary assumptions and constraints are given below.

\begin{dfn}[Additional Constraints and Assumptions]
\label{dfn:genericassumptions} ~\\

\begin{enumerate}
    \item \textbf{Chung-Lu generalization} We impose that the probability that two nodes are connected given their intensity similar to the Chung-Lu model \ba P(i \sim j | \rho_i,\rho_j ) = \min\left\{ \frac{\rho_i \rho_j}{\sum_k \rho_k}, 1\right\} \label{eq:Chunglu}\;.\ea 
    \item \textbf{Compatible expected edge densities} The expected edge density of the graph from the geometric model is given by $\varep_{geom} = \lim_{x \to \infty} F_1(x)$. To get the expected edge density from the Chung-Lu generalization note that \bas \mathbb{E}(\deg_i) = \sum_{j} P(i\sim j| \rho_i, \rho_k) = \sum_j \min\{ \frac{\rho_i \rho_j}{\sum_k \rho_k}, 1  \} \;.\eas Therefore, the edge density of the graph is \bas \varep_{CL} =  \frac{\sum_k \mathbb{E}(\deg_k)}{n^2} \;.\eas 
    The expected edge densities computed by the geometric and by the Chung-Lu models must coincide: \bas \varep_{geom} = \varep_{CL} = \varep \;. \eas We denote by $\varep$ edge density of the model.
    \item \textbf{Independence} The distances between two vertices are independent of their intensities: \ba P(d_{i,j} < x, \rho_i, \rho_j)  = P(d_{i,j} < x) P (\rho_i,  \rho_j ). \label{eq:indep}\ea Furthermore, we require that the distances between two vertices and the expected degrees of the graph remain independent, given that they are adjacent: \ba P(d_{i,j} < x, \rho_i, \rho_j| i \sim j )  = P(d_{i,j} < x | i \sim j )P (\rho_i,  \rho_j | i \sim j ). \label{eq:condindep}\ea
\end{enumerate}
\end{dfn}

Note that the Chung-Lu generalization in Definition \ref{dfn:genericassumptions} is a generalization of the Chung-Lu condition in Section \eqref{eq:CLcondition}. This is the same generalization of this condition adopted by \cite{bringmann2019geometric}. Namely, we allow the intensities to be such that the value of $\frac{\rho_i \rho_j}{\sum_k \rho_k}$ may be greater than 1. When this occurs, we set the probability that the two nodes are connected to $1$. 

Under the conditions laid out in Definition \ref{dfn:genericassumptions} we may calculate the probability $P( d_{i,j} <x )$  for any two randomly chosen points $i,j$ given their positions  in $\mathbb{R}^d / \mathbb{Z}^d$. Call this function $F_2(x)$ given by \ba F_2(x) = P(d_{ij}< x) = \idotsint_{\|\mathbf{y}-\mathbf{z}\|_2 < x}  (\mu \times \mu)(\mathbf{y},\mathbf{z}) ~ dy_1 \ldots dy_d~ dz_1 \ldots dz_d \label{eq:dist}\ea 
where $(\mu \times \mu)$ is the product measure.
Note that the derivative of $F_2(x)$ can be interpreted as a probability density function, while the derivative of $F_1(x)$ cannot. Namely, the derivative $F_1'(x)$ gives the probability that two nodes are connected and distanced exactly $x$ apart: $F_1'(x) = P(i \sim j, d_{i,j} = x)$.


We are now ready to build the connection function for this model.

\begin{lem} \label{res:Pij|distance}
We may find the probability that two nodes are connected given that they are a specific distance apart as
\bas  P(i \sim j | d_{i,j} =x ) = \frac{F'_1(x)}{F'_2(x)} \;.\eas
\end{lem}

\begin{proof}
First note that by Bayes rule, \eqref{eq:distandconnect} and \eqref{eq:dist}, we may write the probability that two nodes are connected given that they are within a given distance as \ba P(i \sim j | d_{i,j} <x ) = \frac{F_1(x)}{F_2(x)} \;. \label{eq::f1overf2} \ea One may calculate the probability that two nodes are connected given that they are exactly a fixed difference apart by taking the derivatives of the numerator and denominator of the above display:
\ba 
\frac{ F_1'(x)}{ F_2'(x)} &= \lim_{h \to 0} \frac{P(i\sim j, x < d_{ij} < x+h )}{P(x < d_{ij} < x+h )} \\ &= \lim_{h \to 0} P(i \sim j | x < d_{i,j} <x + h  ) =  P(i \sim j | d_{i,j}  =  x )\label{eq:congivendistpdf} \ea as desired. 
\end{proof}

\begin{rmk}\label{rmk:F_1(0)behavior}Note that there is some delicacy in how we define the function $F_1(x)$ and its derivatives. Namely, we set it uniformly to $0$ for $x < 0$. However, if there are self loops, this function is no longer continuous. We have \bas F_1(x) = \begin{cases} 0 & x < 0  \\ 
\frac{\mathbb{E}(\# \; \textrm{self loops})}{n^2} & x=0 \\
P(i\sim j, d_{i,j} < x) & x >  0\end{cases}
\eas

When $x = 0$, we have \bas F_1(0) = \frac{P(i \sim i)}{n} = \frac{ \mathbb{E}(\#\; \textrm{self loops})}{n^2} \; .\eas  As a result, under mild conditions on the model, the function $F_1(x)$ is continuous on and differentiable on $(0 , \infty)$ with $\lim_{x \to \infty} F_1(x) = \varep.$ We define \bas F'_1(0) = \frac{\mathbb{E}(\# \textrm{self loops})}{n^2}  \delta(0)  \eas where $\delta(\cdot)$ is the Dirac-delta function.  Thus, $\int_0^\infty F'_1(x) ~{dx} = \varep$.
\end{rmk} 


We may now use Lemma \ref{res:Pij|distance} and the fact that the distance between the nodes is not only independent of their intensity, but that both are conditionally independent on whether the nodes are adjacent to define the connection function. 

\begin{thm} \label{res:edge|dist,valence}
The connection function for this model is \bas P(i\sim j | d_{i,j} = x, \rho_i, \rho_j) = \min\{\frac{\rho_i \rho_j}{\sum_k \rho_k}, 1\} \frac{F_1'(x)}{F'_2(x)} \frac{1}{\varep}\;, \eas  

\end{thm}

\begin{proof}

By Bayes rule, and assumption \eqref{eq:condindep}, we may write \bas P(i\sim j | d_{i,j} = x, \rho_i, \rho_j) P(d_{i,j} = x, \rho_i, \rho_j )  &= P(d_{i,j} = x, \rho_i, \rho_j| i\sim j ) P(i \sim j) \\ &=  P(d_{i,j} = x| i\sim j ), P(\rho_i, \rho_j| i\sim j ) P(i \sim j) \\ & = \frac{P( i\sim j |d_{i,j} = x ) P( i\sim j | \rho_i, \rho_j) P(d_{i,j} = x) P( \rho_i, \rho_j )}{P( i\sim j )} \;, \eas or by \eqref{eq:indep} \bas P(i\sim j | d_{i,j} = x, \rho_i, \rho_j)  = \frac{P( i\sim j |d_{i,j} = x ) P( i\sim j | \rho_i, \rho_j)}{P( i\sim j )}\;.\eas Equations \eqref{eq:Chunglu} and \eqref{eq:congivendistpdf} give expressions for $P( i\sim j |d_{i,j} = x )$ and $P( i\sim j | \rho_i, \rho_j)$ respectively. By construction $P(i\sim j) = \varep$. Combining these, we have \bas P(i\sim j | d_{i,j} = x, \rho_i, \rho_j) = \min\{\frac{\rho_i \rho_j}{\sum_k \rho_k}, 1\} \frac{F_1'(x)}{F'_2(x)} \frac{1 }{\varep}, \eas as desired.

\end{proof}


Note that if we remove the Chung-Lu generalization, i.e. demand that $\frac{\rho_i \rho_j}{\sum_k \rho_k} \le 1$, we see the the intensities are exactly the expected degree of a vertex, given the associated intensity.

\begin{lem}\label{res:expdegrho}
Restricting to the case where the Chung-Lu condition holds, i.e. to when $\frac{\rho_i \rho_j}{\sum_k \rho_k} \le 1$ for all $i, j$, we have that for any vertex $i$, the expected degree of the vertex is equal to its given intensity:

\[ \mathbb{E}(deg_i | \rho_i ) =  \rho_i \]
\end{lem}

\begin{proof}
Note that by definition 
\bas
\mathbb{E}(deg_i \;|\; \rho_i ) = \sum_j {P}(i \sim j | \rho_i ) \eas which, by the law of total probability can be rewritten   
\bas \mathbb{E}(deg_i \;|\; \rho_i )  = \sum_j  \sum_\ell {P}(\rho_j = \rho_\ell) {P}(i\sim j \;|\; \rho_i, \rho_\ell ) \; .\eas Similarly, we may write  $ {P}(\rho_j = \rho_\ell) {P}(i\sim j \;|\; \rho_i, \rho_\ell ) = \int_0^\infty  F_2'(x) {P}(\rho_j = \rho_\ell\;|\; d_{ij} = x) {P}(i\sim j \;|\; \rho_i, \rho_\ell , d_{ij} = x ) ~{dx}$. Therefore, we may write \bas \mathbb{E}(deg_i \;|\; \rho_i ) = \sum_j  \sum_\ell \int_0^\infty  F_2'(x) {P}(\rho_j  = \rho_\ell\;|\; d_{ij} = x) {P}(i\sim j \;|\; \rho_i, \rho_\ell, d_{ij} = x ) ~{dx} \;.\eas
Note that since the $\rho_j$ are distributed randomly amongst the vertices, ${P}(\rho_j = \rho_\ell\;|\; d_{ij} = x) = {P}(\rho_j = \rho_\ell) = \frac{1}{n}$. Combined with  Theorem \ref{res:edge|dist,valence} we rewrite this \bas \mathbb{E}(deg_i \;|\; \rho_i ) = \sum_j  \sum_\ell \int_0^\infty F_2'(x) \frac{1}{n}  \min\{\frac{\rho_i \rho_j}{\sum_k \rho_k}, 1\} \frac{F_1'(x)}{F'_2(x)} \frac{1}{\varep} \;.\eas Since there is not $\ell$ in this statement anymore, by the assumption that $\frac{\rho_i \rho_j}{\sum_k \rho_k} \le 1$,  this simplifies to \bas \mathbb{E}(deg_i \;|\; \rho_i ) = \frac{1}{\varepsilon} \sum_j  \frac{\rho_i \rho_j}{\sum_k \rho_k}  \int_0^\infty  F_1'(x) ~{dx} \;. \eas By Remark \ref{rmk:F_1(0)behavior}, we see that $\int_0^\infty F'_1(x) ~{dx} = \varep$. In other words, \bas \mathbb{E}(deg_i \;|\; \rho_i ) = \sum_j  \frac{\rho_i \rho_j}{\sum_k \rho_k}  = \rho_i \eas as desired.


\end{proof}

\subsection{The generic geometric Chung-Lu model: without uniformity and with boundaries \label{sec:specific}}

While the generic model in Section \ref{sec:genericmodel} is mathematically convenient, the toroidal nature is not present in real-world networks. To address this, we modify the original model by additionally prescribing an addition parameters.

\begin{dfn} \label{dfn:specificinputs}
In addition to the parameters given in Definition \ref{dfn:genericinputs}, we add an additional input, a region within the torus:
\begin{enumerate} \setcounter{enumi}{5} 
    \item A region $\rB \subset \mathbb{R}^d / \mathbb{Z}^d$.
\end{enumerate}
\end{dfn}

This restriction to a region induces several changes to the parameters given in Definition \ref{dfn:genericinputs}.

First, we need to change the graph itself, as well as the measure on this space. To generate a random graph, we select $n$ points within $\rB$ using the measure $\frac{1}{\mu(\rB)} \mu$ and generate edges according to the probabilities Theorem \ref{res:edge|dist,valence}.  We then, consider the resulting induced subgraph of the vertices in $\rB$. We call the resulting network $G_{\rB}$.

Second, imposing a region $\rB$ changes the previously defined probabilities $F_1(x)$ and $F_2(x)$. 

\begin{dfn}
Let $F_{\rB,1}(x) = P(i \sim j, d_{i,j}\leq x | i,j \in \rB)$ and $F_{\rB,2}(x) d_{i,j}\leq x | i,j \in \rB)$.
\end{dfn}


As in Lemma \ref{res:Pij|distance}, we write \bas \frac{F_{\rB,1}(x)}{F_{\rB,2}(x)} = P(i \sim j | d_{ij} < x \; ;  i,j \in \rB)\;. \eas Then we may relate this conditional probability to the conditional probability in the generic model, $P(i \sim j | d_{ij} < x)$.

To understand the bias between $F_1(x)$ and $F_{\rB,1}(x)$ and $F_2(x)$ and $F_{\rB,2}(x)$, we define two functions, $G_1(x) = P(i,j \in \rB \;|\; i\sim j, d_{ij}<x)$, and  $G_2(x) = P(i,j \in \rB \;|\; d_{ij}<x)$

\begin{thm} 
\[ \frac{F_{\rB,1}'(x)}{F_{\rB,1}'(x)} = \frac{G_1'(x)}{G_2'(x)}  \frac{ F_1'(x)}{ F_2'(x)}. \]
\end{thm}

\begin{proof}
This proof follows from the arguments presented in Lemma \ref{res:Pij|distance}. Namely 
\ba
\frac{F_{\rB,1}'(x)}{F_{\rB,2}'(x)} &= P(i \sim j | d_{ij} = x \; ;  i,j \in \rB) \label{eq:F1/F2estimate}\ea while by a similar calculation, \bas \frac{G_1'(x)}{G_2'(x)} &= \frac{P(i,j \in \rB \;|\; i\sim j,  d_{ij}=x)}{P(i,j \in \rB \;|  d_{ij}=x)} \;.\eas The result follows from Bayes rule.


\end{proof}

The ratio $G_1'(x)/G_2'(x)$ measures bias introduced to the conditional probability $\frac{F'_1(x)}{F'_2(x)} = P(i\sim j | d_{i,j} = x)$ caused by the shape of $\rB$. 
If the ratio $\frac{G_1'(x)}{G_2'(x)}$ is approximately 1 for all $x$, that is, if the probability that two nodes are connected and the probability that they are both in $\rB$ are conditionally independent on their distance, then   
\[ \frac{F_{\rB,1}'(x)}{F_{\rB,2}'(x)} \approx  \frac{ F_1'(x)}{ F_2'(x)}. \] 


Finally, by restricting the region, one must also change the intensities of the vertices accordingly.

\begin{dfn} \label{dfn:rhorestricted}
Let $\vec \rho_{\rB}$ be the vector of intensities on the graph $G_\rB$. Namely, as a generalization of Lemma \ref{res:expdegrho}, $\vec \rho_{\rB} = \mathbb{E}(deg_i^\rB | \rho_i,  d_{i,j})$.
\end{dfn}

We may relate this quantity to the original vector of intensities. However, in order to do so, we first need to define an expected valency of each vertex given its position in the graph, using Lemma \ref{res:Pij|distance}.

\begin{dfn} 
As above, write $d_{i,k}$ to indicate the distance between the vertex $i$ and $k$. Let $\omega_i = \sum_j P(i \sim j | d_{i,j}, i, j \in \rB)$ be the geometric weight of vertex $i$, or the expected degree of the vertex, given its position relative to all other vertices. 
\end{dfn}

Note that $\omega_i$ is a real number associated to each vertex, and not a function of $x$.

\begin{lem} \label{res:geomweight}
Let $d_{i,j}$ be the distance between vertex $i$ and vertex $j$, while $\deg_j$ is the degree of the vertex $j$. Then the geometric weight of a vertex given its position can be expressed \bas \omega_i= \sum_j \frac{F_{\rB,1}'(d_{i,j})}{F_{\rB,2}'(d_{i,j})}\;.\eas
\end{lem}
\begin{proof} 

Summing both sides of equation \eqref{eq:F1/F2estimate}, we see that \bas \sum_i P(i \sim j | d_{i,j} =x; \; i, j \in \rB ) = \sum_i \frac{F_{\rB,1}'(x)}{F_{\rB,2}'(x)} \;.\eas Evaluating this at the distances between the nodes as they have been distributed in $\R^d / \Z^d$, we may rewrite this as \bas \omega_i = \sum_j P(i \sim j | d_{i,j} ; \; i, j \in \rB) = \sum_j \frac{F_{\rB,1}'(d_{i,j})}{F_{\rB,2}'(d_{i,j})} \eas 
\end{proof}

We can use this to calculate $\vec \rho_{\rB}$, the intensity vector in the bounded region.


\begin{thm} \label{res:rhohat}
The expected degree of a vertex $i$ on the induced subgraph in the region $\rB$, given $\rho_i$ and its distance from all points, $d_{i,j}$, is 
\[ \rho_{\rB, i} = \rho_i \frac{\omega_i}{n \varepsilon}  \]
\end{thm}

\begin{proof}
If its intensity is known, as well as the position of the other nodes in the graph. Therefore, we may write 
\begin{eqnarray} 
\mathbb{E}(deg_i^\rB | \rho_i, d_{i,j}) &=&  \sum_{j \in \rB} P(i \sim j | \rho_i, d_{i,j}) \\  & = &\sum_{j \in \rB} \sum_\ell P(\rho_\ell |\rho_i, d_{i,j} ) P(i \sim j | \rho_i, \rho_j=\rho_\ell, d_{i,j}) \end{eqnarray} where the second equality comes from the marginal law of probabilities.

By construction, $P(\rho_\ell |\rho_i, d_{i,j} ) = \frac{1}{n}$. Therefore, \bas 
\mathbb{E}(deg_i^\rB | \rho_i, d_{i,j}) = \frac{1}{n}\sum_{j \in \rB} \sum_\ell  P(i \sim j | \rho_i, \rho_j=\rho_\ell, d_{i,j})
\eas 

By substituting the restricted functions $F_{\rB,1}(x)$ and $F_{\rB,1}(x)$ into Theorem \ref{res:edge|dist,valence} and evaluating at the distances between nodes, the right hand side becomes
\bas 
\mathbb{E}(deg_i^\rB | \rho_i, d_{i,j}) = \frac{1}{n} \sum_{j \in \rB} \frac{F_{\rB,1}'(d_{ij})}{\hat F_2'(d_{ij})} \frac{1}{\varepsilon} \cdot  \min\{\sum_\ell \frac{\rho_i \rho_\ell }{\sum_k \rho_k}, 1\} \eas Note that $\sum_\ell \frac{\rho_i \rho_\ell }{\sum_k \rho_k} = \rho_i$. Therefore, by Lemma \ref{res:geomweight}, we get 
\bas \mathbb{E}(deg_i^\rB | \rho_i, d_{i,j}) = \omega_i \frac{1}{n \varepsilon} \cdot \rho_i \;.\eas Definition \ref{dfn:rhorestricted} gives the desired statement.
\end{proof}

\subsection{Estimating the parameters given a network\label{sec:estimatedmodel}}

In many cases of real world data, one does not actually have the parameters needed to fit the model described above. Rather, they must be estimated from the data observed in a few given networks. 

In this section, given a fixed network $G$, we estimate the parameters given in Definitions \ref{dfn:genericinputs} and \ref{dfn:specificinputs}. In Section \ref{sec:connectomesimulation} we used these estimated values to synthesize networks similar to $G$. We call $G$ the {\it reference graph}.

\begin{dfn}[Inputs from the data]
\begin{enumerate}
    \item The reference graph $G$ with a list of vertices and edges $G= (V, E)$ where $E$ may include self-loops. (Note that this input includes the degree of each vertex, ${deg}_i^G$.)
    \item A set of coordinates $\mathcal{P}$ for the vertices in the reference graph (from which all $d_{i,j}$ can be calculated).
\end{enumerate}
\end{dfn}

\begin{dfn}[Estimated quantities] ~\\
\begin{enumerate} 
\item A function $\hat{F}_1(x)$ fit from the location and adjacency data in the reference graph: $F_{\rB, 1}(x) = P(i\sim j \;, d_{i,j}<x| i,j \in\rB)$.

\item A function $\hat F_2(x)$ fit from the empirical cumulative probability function of the geometric distribution of the vertices of the reference graph: $F_2(x) = P( d_{i,j}<x| i,j \in\rB)$. 

\item A vector of intensities for each vertex, $\hat \rho_i$ to be calculated according to Theorem \ref{res:rhohat}.
\end{enumerate}
\end{dfn}

To calculate $\hat F_1(x)$, $\hat F_2(x)$, we simply fit the corresponding equally-spaced quantiles based upon the inputs of $G$ and $\mathcal{P}$. To calculate $\hat \rho$, we may write $\deg_i^G$ as a point estimate for $\mathbb{E}(deg_i^\rB | \rho_i, d_{i,j})$. By rearranging Theorem \ref{res:rhohat} may take the estimate the vector of intensities, $\rho$, as 
\ba \hat\rho_i =  \frac{deg_i^G \cdot n\varepsilon}{w_i}. \label{eq:rhohat} \ea
In other words, the estimator for $\rho_i$ is proportional to the ratio between the degree of the references graph and the degree one would expect based on the geometries of the nodes. 

There are several remarks to note about these estimates.

First, $\rho_i$ (and $\hat \rho_i$) represent the expected degree in the full unbounded model; hence, one should expect that $\hat \rho_i$ is larger than the corresponding degree in the reference graph, $\deg^G_i$.

Second, note that the numerator of $\hat F_1$ is $\varepsilon$. From Definition \ref{dfn:genericinputs}, one expects the limit \bas \lim_{x \to \infty} F_1(x) = \lim_{x \to \infty} P(i\sim j , d_{i,j}<x| i, j \in \mathcal{B}) = \varepsilon\eas to be the edge density of the graph. Similarly, 
the numerator of $\hat F_2$ is $1$. From Definition \ref{dfn:genericinputs}, one expects the limit \bas \lim_{x \to \infty} F_1(x) = \lim_{x \to \infty} P(d_{i,j}<x| i, j \in \mathcal{B}) =  1\eas to be 1, as the volume of $\mathcal{B}$ is finite.


Further, one may notice that we omit several of the more trivial parameters of the model as $n$ and $d$ are subsumed within the data given by $\mathcal{P}$. On the other hand, $\mu$, is not taken as an input. Instead, given a reference graph $G$, we take $\mu$ to be the point-masses induced by the locations of the vertices given by $\mathcal{P}$. That is, we fix the locations of the vertices of the input data. 

\section{Simulated Drosophila Medulla connectome \label{sec:connectomesimulation}}

In this section, we consider two graphs arising from the connectome of the Drosophila Medulla and fit the parameters of the model described in section \ref{sec:geommodel}. 

\subsection{Drosophila Medulla Connectome \label{sec:originalconnectome}}

As an example graph to validate this model, we consider a connectome of the visual system of the Drosophila medulla consisting of 1781 different neurons with 33,641 synapses \cite{neurodata}.  Note that this is not the full Drosophila medulla, which has approximate 40,000 neurons \cite{generaldrosophilla}. Instead, this is a mapping of a physically extracted sample of full medulla. Briefly, the authors of \cite{naturenetwork} consider seven neuron columns important for motion detection. The neurons in the central column are traced, that is, all of their connections to neurons in neighboring six columns are recorded. Each synapse is given a pair of locations, a pre-synaptic and and post-synaptic position, corresponding to the two halves of the connection corresponding to the two neurons involved.

Biologically speaking, synapses take up small volumes, while the neurons themselves can be large with many long, spindly dendrites that branch across a large volume and intertwine with dendrites of other neurons. Each synapse represents a directed edge between the neurons (from the pre-synaptic neuron to the post synaptic neuron), where a pair of neurons may share multiple directed connections between each other, in either direction. Occasionally neurons have synapses with themselves.

In order to validate this model, we represent the neurons as the nodes of the graph while the synapses are represented by edges. We further simplify the connectome by ignoring the directions and multiplicities of the edges between the nodes. That is, we do not keep track of how many synapses connect one neuron to another, or in which direction the synapse allows for current flow. We only track whether or not two neurons have a connection between them. We do, however, keep track of when a neuron is connected to itself. That is, we allow for self loops in the graph. 

Finally, as we are studying a geometric model, we assign positions to the neurons. Each synapse in the connectome is given a pre-synaptic and a post-synaptic coordinate, as well as a pre-synaptic neuron and a post-synaptic neuron. On a directed multi-graph, one may think of this as each directed edge being decorated by a pair of coordinates (pre and post coordinates). To assign a single coordinate to a neuron, we take the geometric centroid of the synaptic coordinates associated to it. I.e. if we were to model this as a directed multi-graph, we take the centroid of the pre-synaptic coordinates of all the edges from a node, and all the post-synaptic coordinates of all the edges to the node. In this manner, we define a point position to a fundamentally three-dimensional, non-convex volume filling object. Note that we have no reason to believe that the coordinate thus assigned to the neuron lies in the physical neuron in the medulla, let alone reason to believe that it corresponds to any meaningful part of the neuron, such as the location of the soma, or the center of mass of the neuron. However, given the data at hand, and the intrinsic difficulty of mapping the geometry of all the neurons in a connectome, we use this as an approximation of the ``location'' of each neuron, see Figure \ref{fig:distributionofneurons}. Figure \ref{fig:graphofconnectome} is a graphic of the connectome described here.

\begin{figure} [h]
  \centering
  \begin{subfigure}[b]{0.3\linewidth}
    \includegraphics[width=\linewidth]{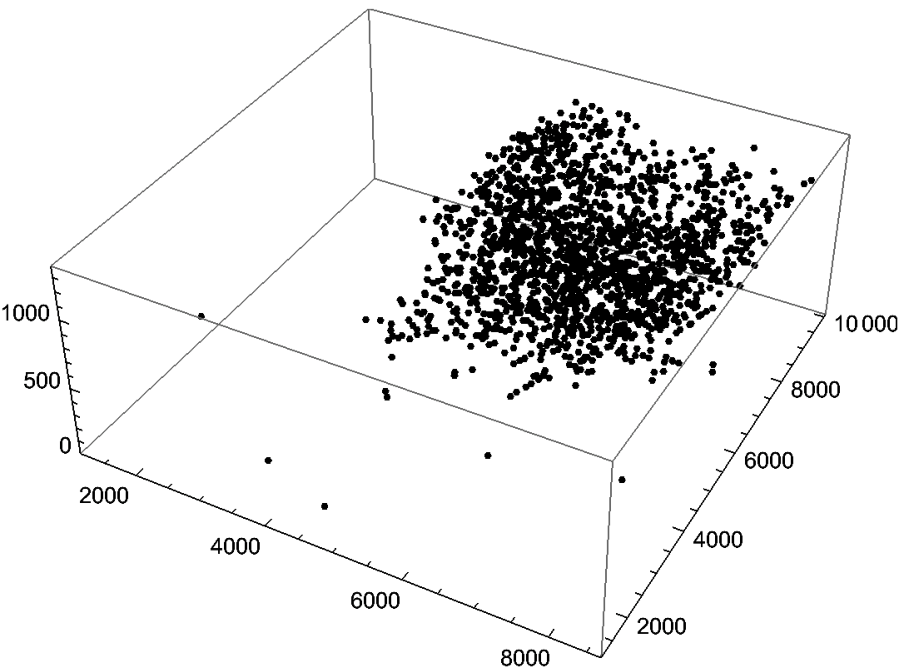}
  \end{subfigure}
  \begin{subfigure}[b]{0.3\linewidth}
    \includegraphics[width=\linewidth]{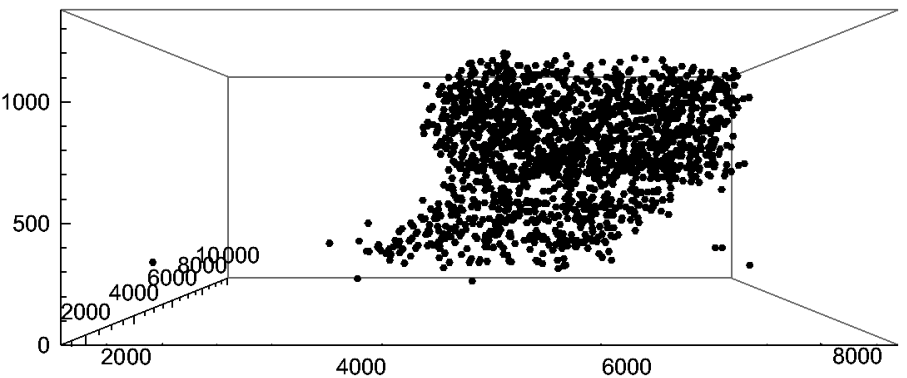}
  \end{subfigure}
  \begin{subfigure}[b]{0.3\linewidth}
    \includegraphics[width=\linewidth]{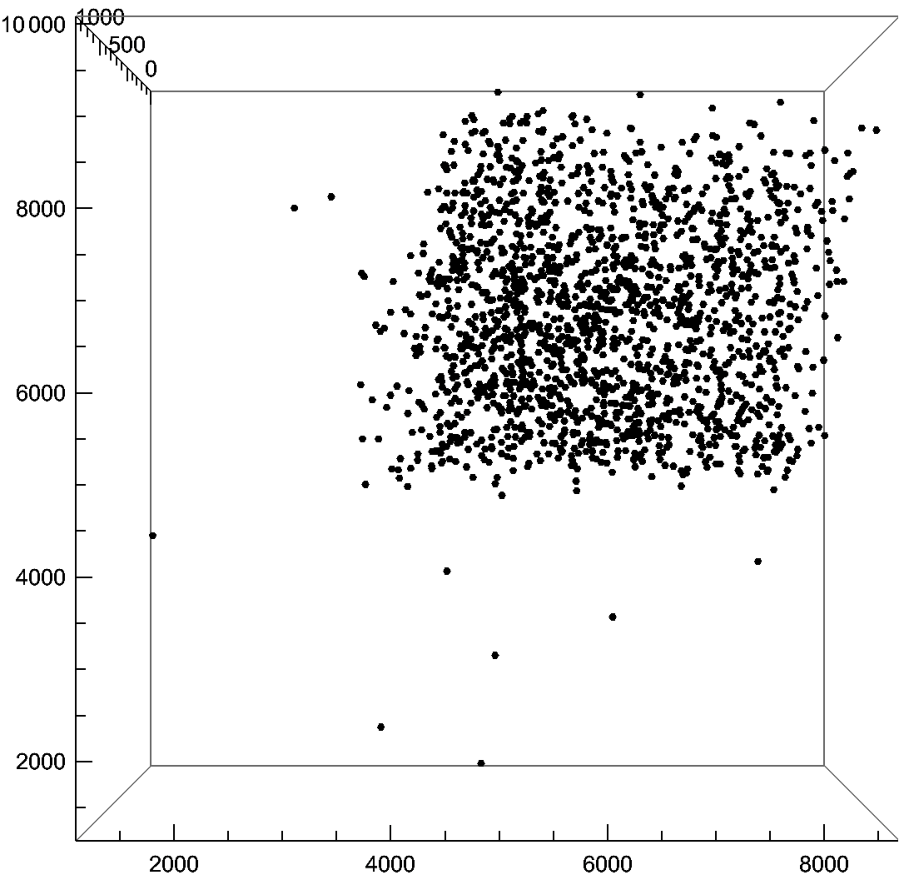}
  \end{subfigure}
  \caption{Three views of the graph of the connectome of the Drosophila Medula, with the vertex locations computed as the centroid of the synapse locations. Note that the dimensions of the bounding box of the connectome is $7290 \; \mu m \times 8685 \; \mu m \times 1212 \; \mu m$. Note that $98\%$ of the data lies in a bounding box of dimensions $4300 \; \mu m \times 4005 \; \mu m \times 1130 \; \mu m$.  } 
  \label{fig:distributionofneurons}
\end{figure}

\begin{figure} [h]
  \centering
   \includegraphics[scale = .5]{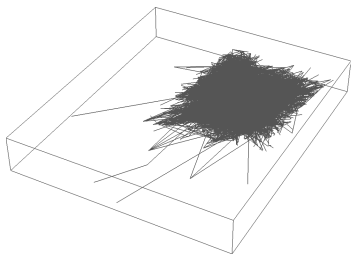}
   \caption{The connectome of the of the Drosophila Medula connectome}
  \label{fig:graphofconnectome}
\end{figure} 

In summary, we model the connectome found in \cite{neurodata} as a simple graph with self loops with $1781$ and $9016$ edges, $105$ of which are self loops. However, there is a further wrinkle that must be explored. As stated in \cite{naturenetwork}, the authors trace the $379$ neurons in the central neural column extracted from the Drosophila medulla. That is, all of neighborhoods of these central $379$ vertices are traced, but connections between other neurons are not necessarily traced. In other words, this connectome may be better modeled if seen as a union of stars of the fully traced neurons. 

In the dataset present in \cite{neurodata}, there are $358$ ``named'' neurons (i.e. these are the neurons with alpha-numeric labels) that {correspond roughly to the central $379$ nodes.}\footnote{This correspondence has been verified by the authors of \cite{neurodata}.}. In aggregate, the location (as determined above) of the named neurons are slightly closer to the centroid than all vertices. Furthermore, the mean degree of the named neurons is 30.05 versus 10.01 for all neurons.

Therefore, in this paper, we consider two graphs: the \emph{full connectome} and the \emph{named subgraph}. The full connectome is the graphical representation of the entire data set (as described above). The named subgraph is the induced subgraph of the full connectome defined by the $358$ named vertices. The named subgraph has 3254 edges, 72 of which are self loops. For our purposes, each of these networks has a select few nodes with an overwhelmingly high degree, violating the Chung-Lu condition. For simplicity, we study the networks with these outlier nodes removed which we denote $\namedtwo$ and $\fullone$.

\subsection{Justification for a geometric model \label{sec:geomjustification}}

The generic geometric Chung-Lu model presented in this paper is particularly well suited for creating synthetic models for a connectome. In particular, neurons are more likely to connect to those closer than farther from it, indicating that there is a geometric component to the formation of this network. On a mesoscale (studying the connections between brain regions roughly the size of the entire graph studied in this paper) there is successful work modeling the strength of connections between brain regions using a continuous spatial propagation model \cite{fieldtheory}. Furthermore, different types of neurons have different numbers of synapses, but this number is by no means fixed. In fact, there is evidence that the these connections change over an organisms lifespan \cite{rewiring}. Furthermore, different neurons of the same type having different valencies, for reasons that are not fully biologically understood. This can be most easily seen in the small Caenorhabditis elegans connectome \cite{VarshneyCelegans}. Therefore, it is natural to work with a Chung-Lu type model where one associates an expected degree or intensity to each vertex. 

In our specific context, we make a few more observations. First, for the full connectome minus one vertex, $\fullone$, the mean distance between any two nodes is 2000.40 $\mu$m; in contrast, the mean distance for any edge 963.25$\mu$m. This substantial discrepancy indicates a strong relationship between the physical distance between two neurons and the likelihood of a connection between them. Also, as illustrated in Figure \ref{fig:indvcumedge}, the distributions of edge lengths appear to be similar but shifted based upon the valency of the node. This suggests that both the distance between two nodes as well as their degrees (i.e., intensities) should play a role in the connection function. Figure \ref{fig:nthneighbors} shows the neighborhood of the first, hundredth and two hundredth most valent vertex in the connectome of the Drosophila medulla graph. 

\begin{figure} [h]
  \centering
  \begin{subfigure}[b]{0.3\linewidth}
    \includegraphics[width=\linewidth]{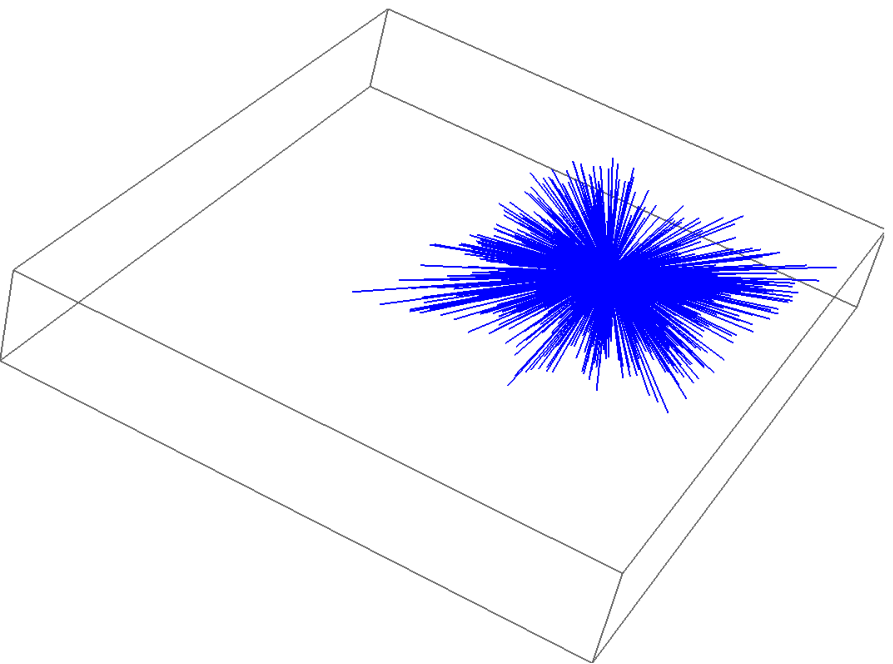}
    \caption{Neighorbohood of the vertex with the highest valence.}
  \end{subfigure}
  \begin{subfigure}[b]{0.3\linewidth}
    \includegraphics[width=\linewidth]{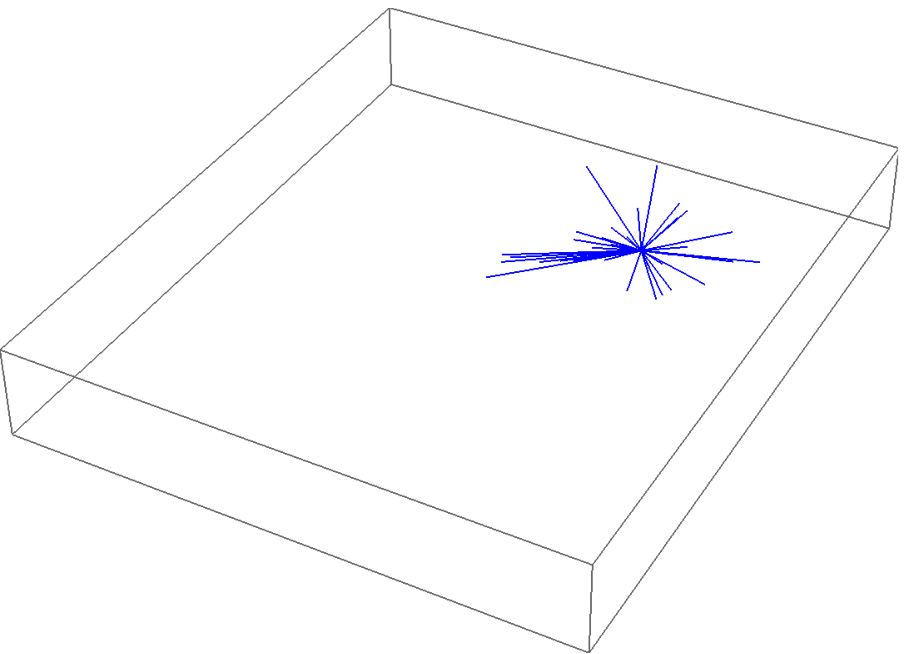}
    \caption{Neighorbohood of the vertex with the $100^{th}$ highest valence.} 
  \end{subfigure}
  \begin{subfigure}[b]{0.3\linewidth}
    \includegraphics[width=\linewidth]{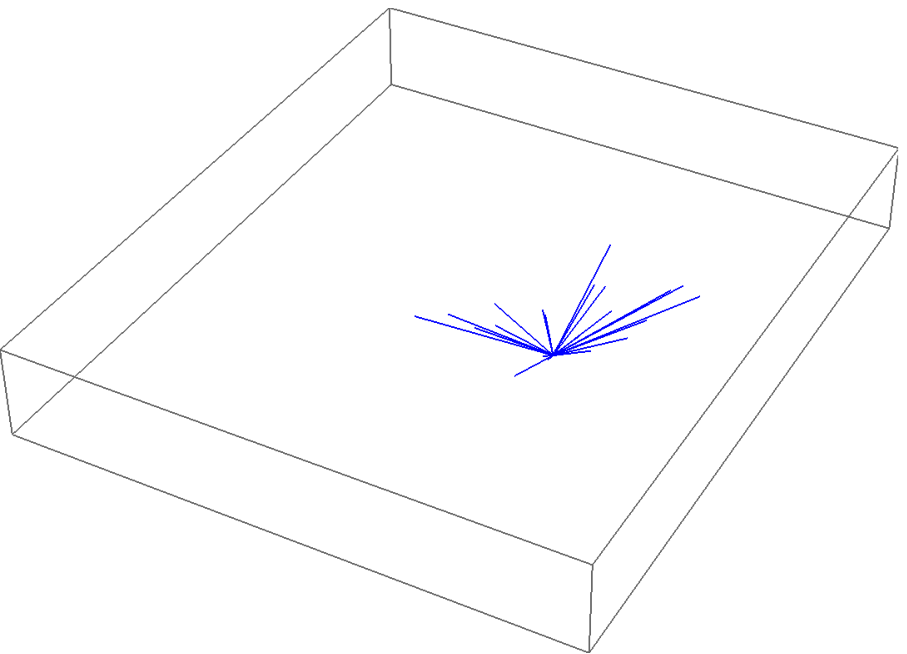}
    \caption{Neighorbohood of the vertex with the $200^{th}$ highest valence.}
  \end{subfigure}
  \caption{Note that while the diameter of the largest connected component in Figure \ref{fig:graphofconnectome} has diameter 6, the neighborhoods of the vertices with the highest, $100^{th}$ highest and $200^{th}$ highest valencies are 927, 37 and 23 respectively, suggesting a correlation between the valency of a vertex and is average distance from neighbors.}
  \label{fig:nthneighbors}
\end{figure}  

\begin{figure} [h]
  \centering
  \begin{subfigure}[b]{0.3\linewidth}
    \includegraphics[width=\linewidth]{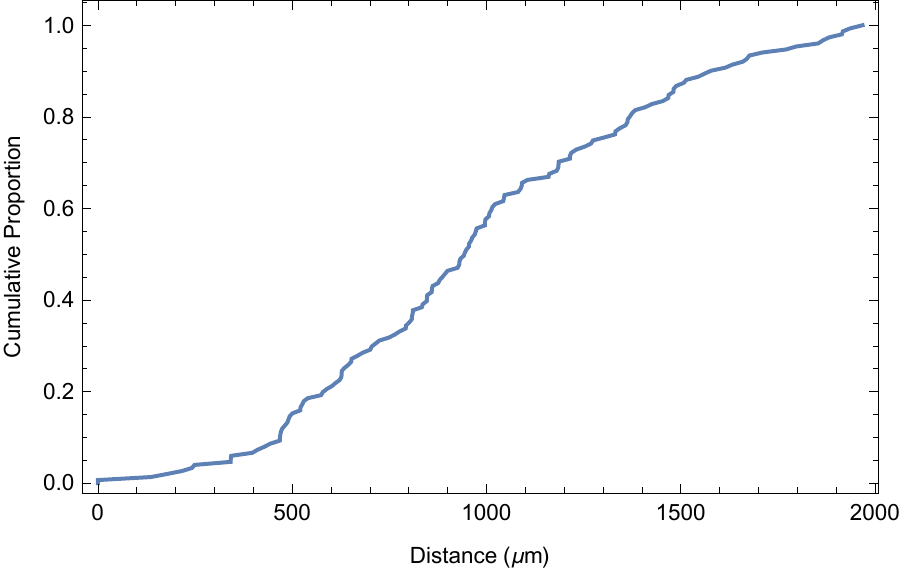}
    \caption{Empirical cumulative distribution of edge distances for the highest valency node in the full connectome minus a vertex.}
  \end{subfigure}
  \begin{subfigure}[b]{0.3\linewidth}
    \includegraphics[width=\linewidth]{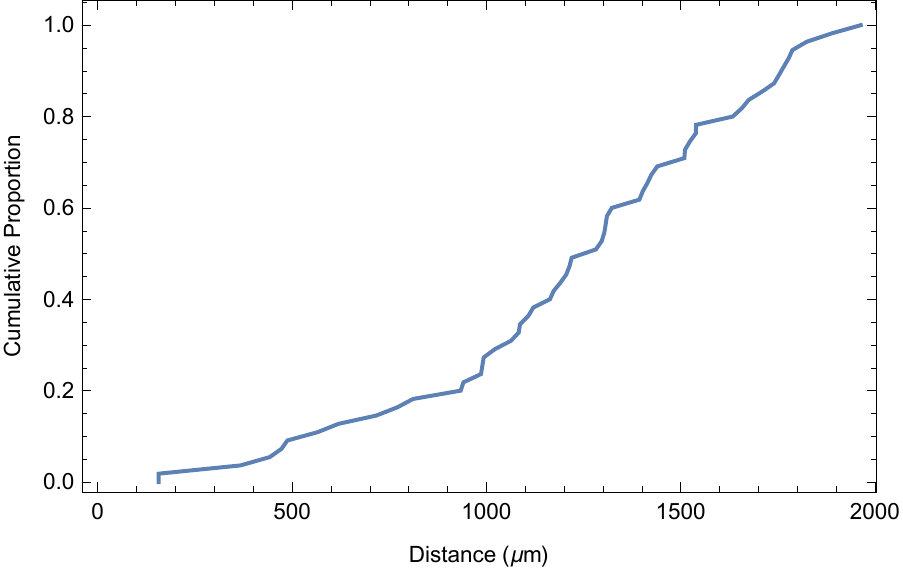}
    \caption{Empirical cumulative distribution of edge distances for the 50-$th$ highest valency node in the full connectome minus a vertex.} 
  \end{subfigure}
  \begin{subfigure}[b]{0.3\linewidth}
    \includegraphics[width=\linewidth]{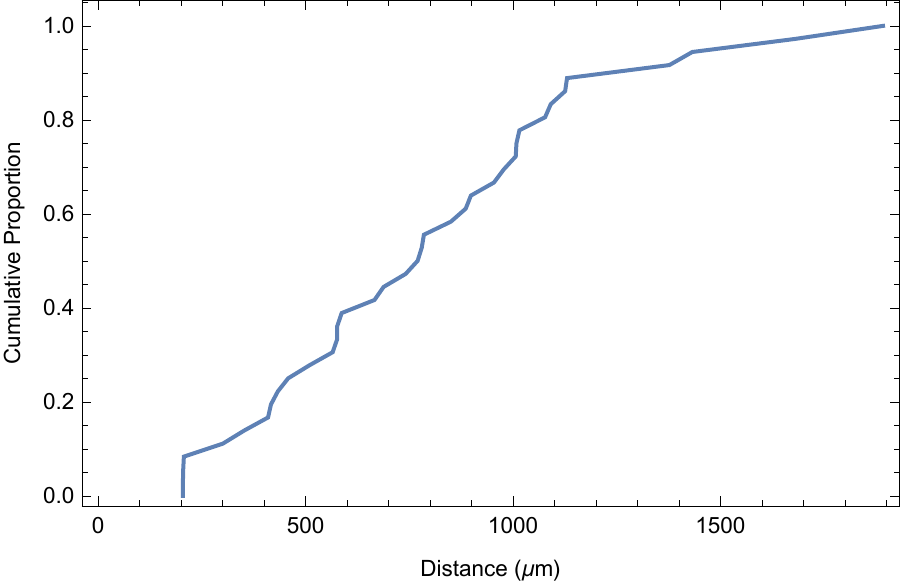}
    \caption{Empirical cumulative distribution of edge distances for the 100-$th$ highest valency in the full connectome minus a vertex.}
  \end{subfigure}
  \caption{The empirical cumulative distribution of edge distances have similar `S' shaped in spite of its valence.} 
  \label{fig:indvcumedge}
\end{figure}  

Finally, it is worth noting that the last assumption of our model, that the distances and intensities of our vertices are independent is not a biological observation. It is an assumption that is needed for mathematical simplicity. In fact, there is biological evidence that this assumption may not hold, For the Wistar rat, the number of synaptic connections vary greatly between different layers of the brain \cite{synapesedistribution}. It is also well known that the neurons in a brain have non-uniform physical structure. For example, this is explored in detail for the Drosophila Medulla in \cite{generalconnection}.


\subsection{Estimated Parameters for the Drosophila Medulla connectome \label{sec:modelparams}} 




 In this section, we fit the parameters presented in section \ref{sec:estimatedmodel} for the two graphs $\namedtwo$ and $\fullone$ defined in section \ref{sec:originalconnectome}. 
 
First, we note that after fitting $\hat F_1$ and $\hat F_2$ on the named subgraph $G_{\textrm{named}}$ and the full connectome, $G_{\textrm{full}}$ the Chung-Lu condition $\hat \rho_i^2 < \sum_k \hat \rho_k$ does not hold for the two highest values of $\hat \rho_i$ for the named subgraph, and the highest value for the full graph. One potential solution to this is to follow  \cite{bringmann2019geometric}, and set any probability greater than 1 generated in this manner to 1. Another solution (the one adopted in this paper) is to omit the vertex with highest degree from our calculations. For the named subgraph, we disregard the 2 most valent vertices and call the graph $\namedtwo$. For full graph, we omit the single most valent vertex and call the graph $\fullone$. After these omissions, both $\namedtwo$ and $\fullone$ satisfy the Chung-Lu condition. 

Since our goal is to generate a new simplified model that empirically reproduces many features of the underlying data based on a single instance, it is not appropriate to formally estimate parameters with confidence intervals nor test a null hypothesis. Instead, the parameters learned from the data are a necessary step towards creating the connection function in the model presented in section \ref{sec:estimatedmodel}.

Since there is no posited true simple model, there is no null hypothesis to reject and no substantive interpretation of any ``standard errors'' that could be calculated, and thus they are not reported. Instead, to validate our parameters, we present the mean squared error of the observed data from the fit functions. 

In particular, we wish to estimate the probability $P(i \sim j, d_{i,j} | \rB)$ logistically with the function \bas \hat F_1 =   \frac{\varep}{1 + \exp(\alpha_1  + \beta_1x)}\eas and the probability $P(d_{i,j} | \rB)$ with the logistic function \bas \hat F_2 = \frac{1}{1 + \exp(\alpha_2  + \beta_2x)}\;.\eas 

Note that we write the logistic functions above as $\frac{L}{1 + \exp(\alpha + \beta x)}$, and not $\frac{L}{1 + \exp(-k(x - x_0))}$ which is not the standard notation. In the notation of this paper, $-\beta$ give the growth rate of the function. In particular, $\beta$ needs to be a negative number, while the logistic function is centered at $x = -\frac{\alpha}{\beta}$.

In Table \ref{table:parambounds}, we report the estimated values for the constants $\alpha_1$ and $\beta_1$ that define $\hat F_1(x)$ and the constants $\alpha_2$ and $\beta_2$ that define $\hat F_2(x)$. 



\begin{table}[h!]
\centering
\begin{tabular}{|c | c || c |} \hline
 & $\namedtwo$ & $\fullone$ \\ \hline
$\alpha_1$ & 2.94589 &  2.76129 \\ \hline
$\beta_1$ & -0.00293112  & -0.00300996 \\ \hline
$\alpha_2$ & 3.16284  & 3.31531 \\ \hline
$\beta_2$ & -0.0018741  & -0.00170436  \\ \hline
\end{tabular} \\

\caption{Estimated parameters defining $\hat F_1$ and $\hat F_2$ using the reference graphs $\namedtwo$ and $\fullone$. }
\label{table:parambounds}
\end{table} 

Next we consider the estimated and observed values at $x = 0$ in Table \ref{table:limiting values}. Recall from Remark \ref{rmk:F_1(0)behavior} that that $\lim_{x \to 0^+}F_1(0) = \frac{|\textrm{self loops}|}{n^2}$. This leads to the observed values of 0.00056022 for $\namedtwo$ and 0.0000331398 for $\fullone$. However, the estimated values from the model are an order of magnitude higher: 0.0020432 for $\namedtwo$ and 0.000269171 for $\fullone$. Furthermore, the observed value for $F_2(0) = 1/n$, or the probability that the node $i$ is the same as the node $j$.  This leads to the observed values of 0.00280899 for $\namedtwo$ and 0.000561798 for $\fullone$. Again, the estimated values are orders of magnitude larger than the observed values.
\begin{table}[h!]
\centering
\begin{tabular}{|c |c | c || c |c |} \hline
 & \multicolumn{2}{c||}{$\namedtwo$} & \multicolumn{2}{c|}{$\fullone$} \\ \cline{2-5}
 & estimated & observed  & estimated & observed  \\ \hline
$\lim_{x \rightarrow 0} \hat F_1(x)$ & 0.00224016 &  0.00056022 &  0.000303395 & 0.0000331398  \\ \hline
$\lim_{x \rightarrow 0} \hat F_2(x)$ & 0.0405882  & 0.00280899 &  0.0350498 & 0.000561798   \\ \hline
\end{tabular}
\caption{Model-estimated and observed limiting values of the functions $\hat F_1$ and $\hat F_2$ for loops and edges of small length. The impact of these discrepancies are shown in Figure \ref{fig:df1df2fits}.} 
\label{table:limiting values}
\end{table} 

In spite of these inconsistencies, we show in Section \ref{sec:edgesetc.} that the synthetic graphs generated by this model match the graphical characteristics of the reference graphs $\namedtwo$ and $\fullone$. Furthermore, we show that the estimated values of $\hat F_1(x)$ and $\hat F_2(x)$ overall do not differ greatly from the observed value. Namely, in Table \ref{table:errors} give the Mean Squared Error (MSE) of the model as well as the mean percent error for $\hat F_1(x)$ and $\hat F_2(x)$ over the entire range of distances, and away from $x = 0$ where the model fit is the worst.

For $\hat F_1(x)$ the mean percent error is less than $5\%$ overall, and less than $1\%$ for values of $x > 500$. The mean percent error for $F_2(x)$ is greater, as indicated by large difference in the estimated vs observed values near $x = 0$ (Table \ref{table:limiting values}). However, we see that away from this region, namely for $x > 500$, the mean percent error falls to less than $5 \%$. There is a similar pattern with the MSE: the overall MSE is greater for F2 than it is for F1, but for both fits, this is dominated by the behavior for low $x$. For $x>0$ the MSE drops by at least an order of magnitude.





\begin{table}[h!]
\centering
\begin{tabular}{|c|c |c | c || c |c |} \hline
& & \multicolumn{2}{c||}{$\namedtwo$} & \multicolumn{2}{c|}{$\fullone$} \\ \cline{3-6}
  & & all $x$ & $x > 500$  & all $x$ & $x > 500$  \\ \hline
\multirow{2}{*}{$\hat F_1(x)$} & MSE & 0.0138996 & 0.000175405  &  0.0201616 & 0.000250721 \\ \cline{2-6}
& Mean Percent error & 4.24187 \%  & 0.822747 \%  &  4.68995 \% &  0.901061 \%   \\ \hline
\multirow{2}{*}{$\hat F_2(x)$} & MSE & 0.0538964 & 0.00856598  & 0.0694437  & 0.0127325  \\ \cline{2-6}
& Mean Percent error & 12.215 \% & 5.32331 \%  &  12.655 \% & 4.9259 \%  \\ \hline
\end{tabular}
\caption{Mean square errors and  Mean percentage errors of $F_1$ and $F_2$ for the graphs $\namedtwo$ and $\fullone$. These fits are good for larger edges.} 
\label{table:errors}
\end{table}

This behavior is clear from Figure \ref{fig:f1andf2fits}, which shows the logistic fits of the cumulative functions $\hat F_1$ and $\hat F_2$ for the graph $\namedtwo$ and $\fullone$.

\begin{figure}
\begin{subfigure}[t]{0.9\linewidth}
    \centering
         \includegraphics[width=0.45\textwidth]{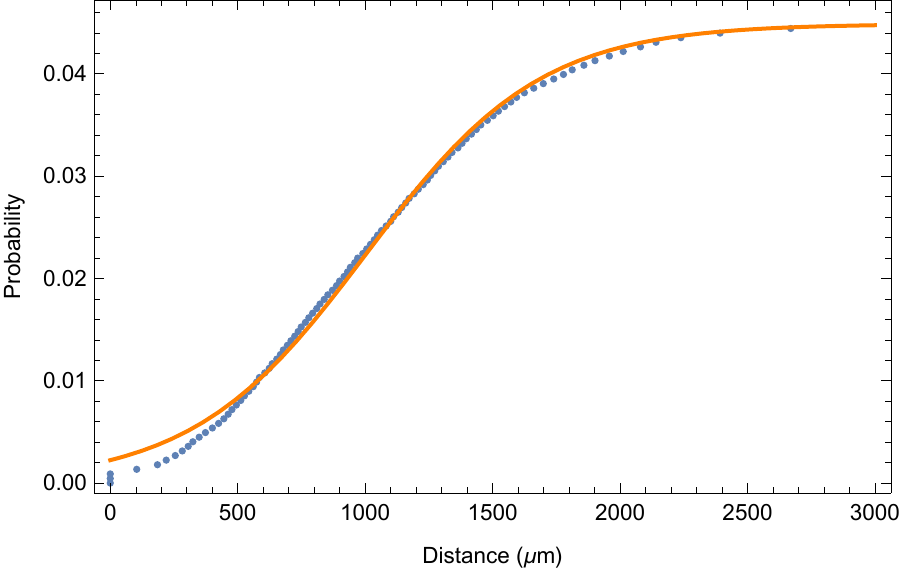}
     \quad \includegraphics[width=0.45\textwidth]{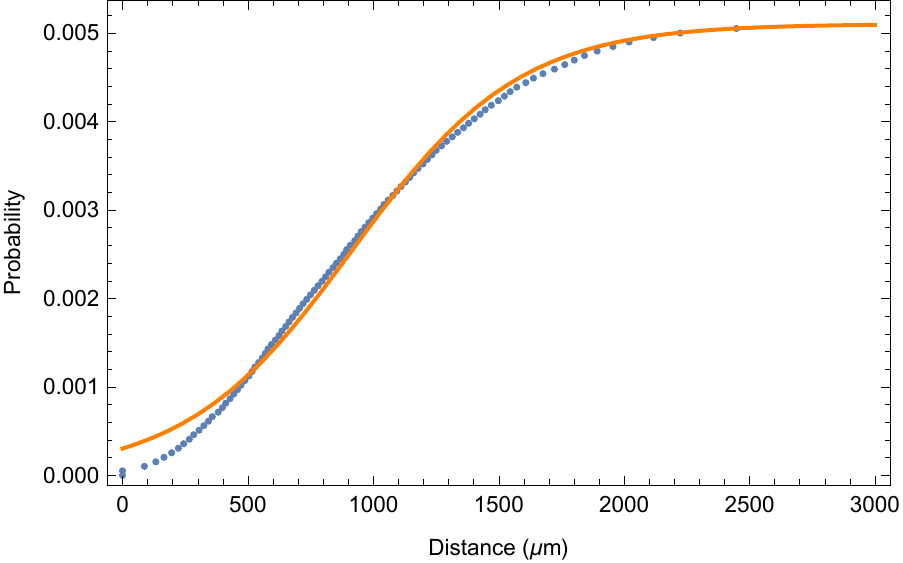}
     \caption{Fits for $F_1(x)$ for $\namedtwo$ (left) and $\fullone$ (right).}
    \end{subfigure}
    
    \begin{subfigure}[t]{0.9\linewidth}
    \centering
         \includegraphics[width=0.45\textwidth]{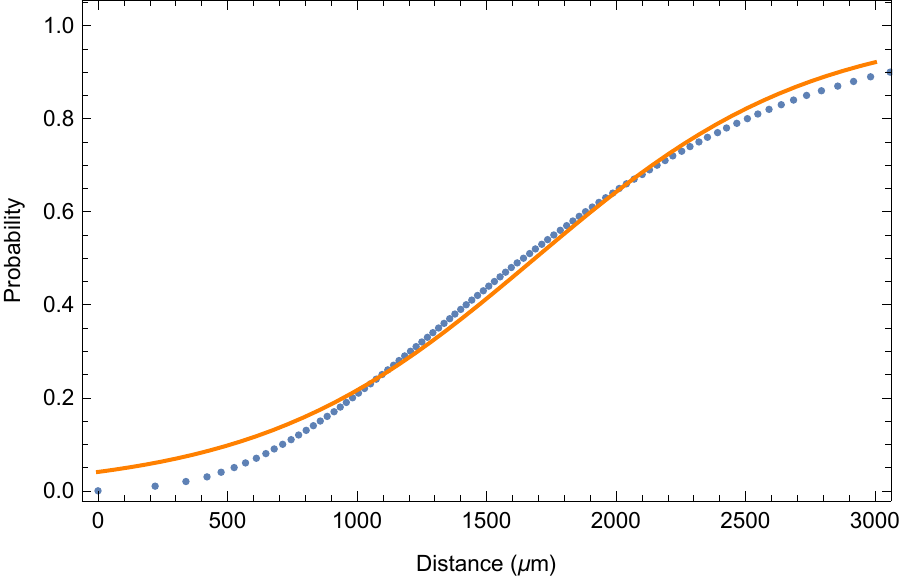}
     \quad \includegraphics[width=0.45\textwidth]{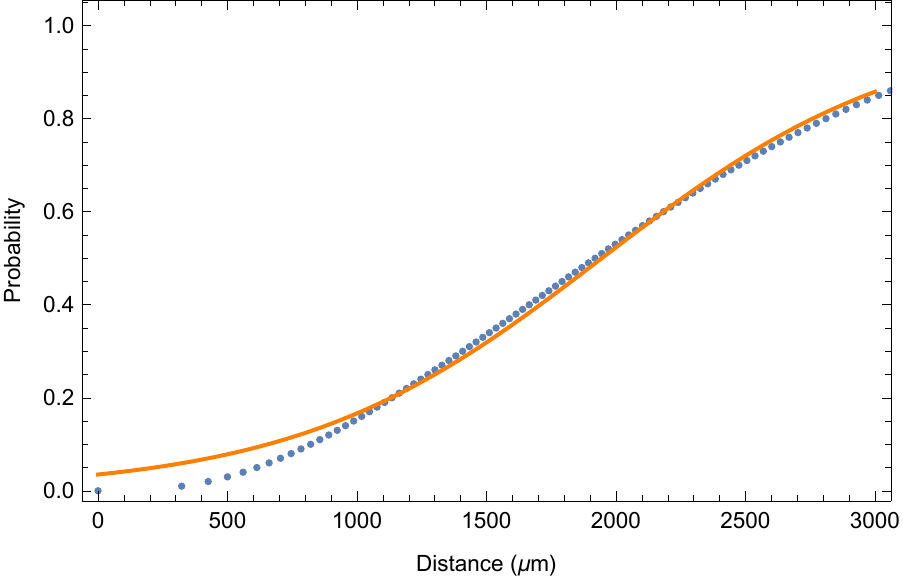}
     \caption{Fits for $F_2(x)$ for $\namedtwo$ (left) and $\fullone$ (right).}
    \end{subfigure}
    
    \caption{The empirical probabilities (blue) for $F_1$ (top) and $F_2$ (bottom) against their modeled probabilities (orange) ($\hat F_1$ and  $\hat F_2$ respectively.). The fits are good, particularly for longer edges.  }
    \label{fig:f1andf2fits}
\end{figure}

Figure \ref{fig:df1df2fits} shows the empirical and estimated ratio $\hat F_1'(x)/ \hat F'_2(x)$, which is central to the connection function defined in Theorem \ref{res:edge|dist,valence}. As with the estimated versus observed values for $\hat F_1(x)$ and $\hat F_2(x)$, we see that the estimated value is a much better fit for large values of $x$ for both $\fullone$ and $\namedtwo$. In other words, the synthetic models generated by these graphs will have fewer edges between nodes that are physically close together than observed in the reference graphs. 

\begin{figure}
    \centering
         \includegraphics[width=0.45\textwidth]{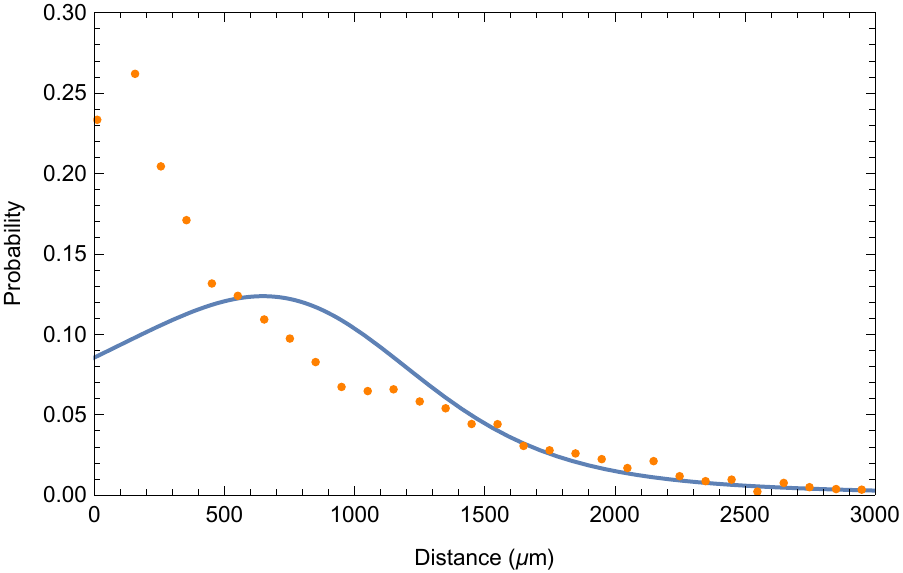}
     \quad \includegraphics[width=0.45\textwidth]{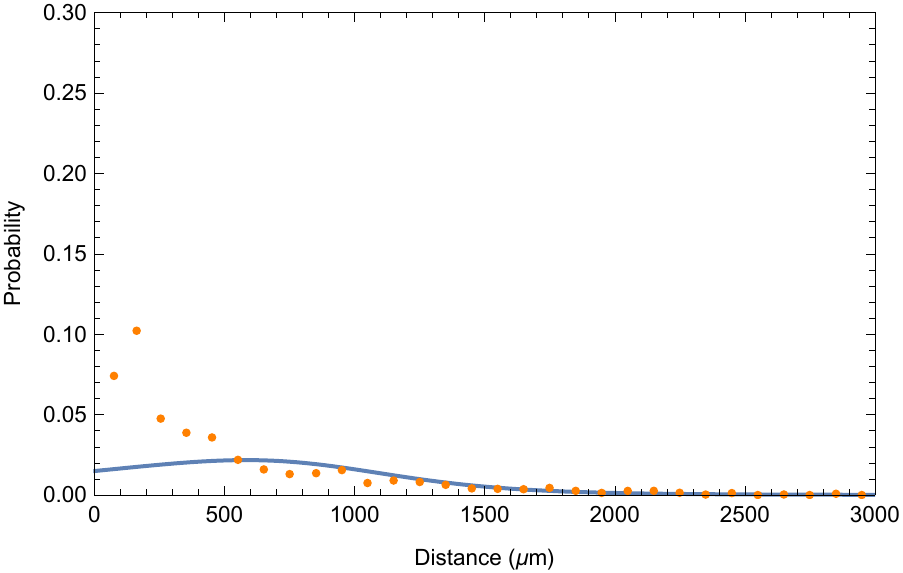}
     \caption{Plot of $\hat F_1'(x)/ \hat F'_2(x)$ (blue) against the empirical probability of an edge at a fixed distance ($P(i \sim j~| d_{i,j}=x)$) (orange). The estimated value is a much better fit for large values of $x$ for both $\fullone$ and $\namedtwo$}

    \label{fig:df1df2fits}
\end{figure}

Finally, Figure \ref{fig:rhodegscatter} addresses the parameters $\hat{\rho_i}$ as defined in display \eqref{eq:rhohat}. The figure shows that there is a linear relationship between the observed degree and the estimated intensity near the centroid of the neurons,  i.e. where the boundary effects laid out in Section \ref{sec:specific} have the smallest influence on the parameters. Furthermore, this proportionality constant near the centroid is less than one. As one moves further from the centroid, i.e. when the boundary of the graph plays a greater role in determining the observed degree of the graph, the estimated intensity increases for any given degree. This is exactly as is expected: the closer a given vertex is to the boundary (i.e. the further it is from the centroid) the more likely it is to have connections to other vertices that are not within the boundary, and thus are cut off. 

\begin{figure}
    \centering
         \includegraphics[width=0.40\textwidth]{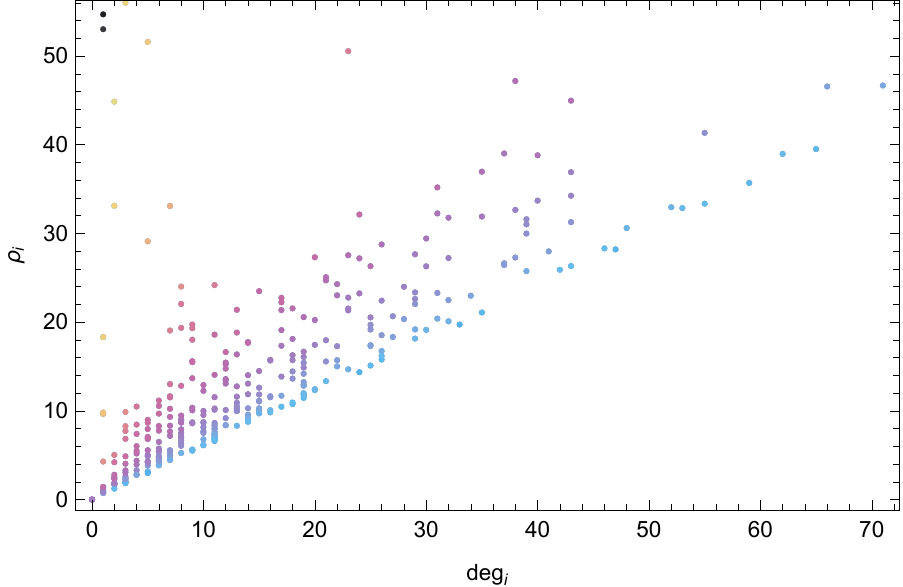}
     \quad \includegraphics[width=0.40\textwidth]{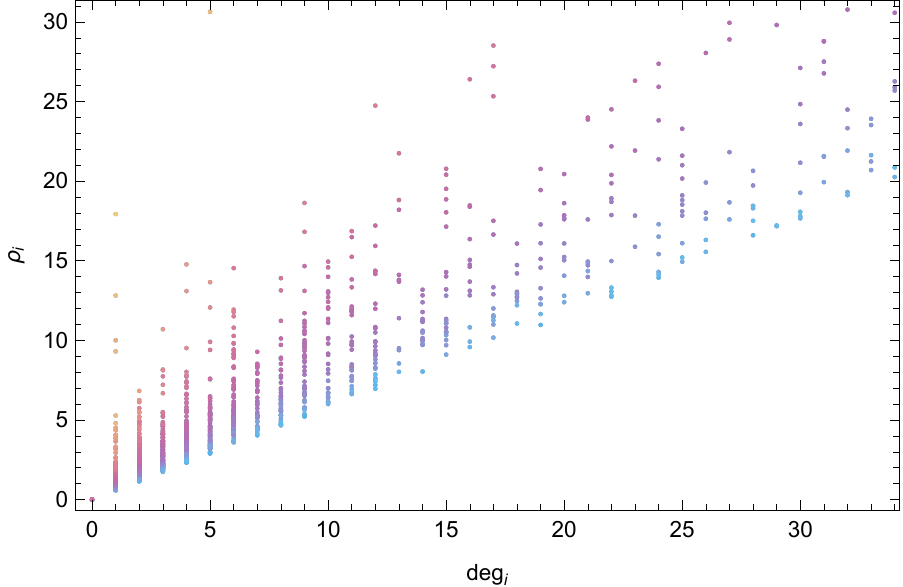}~~\includegraphics[width=.1\textwidth]{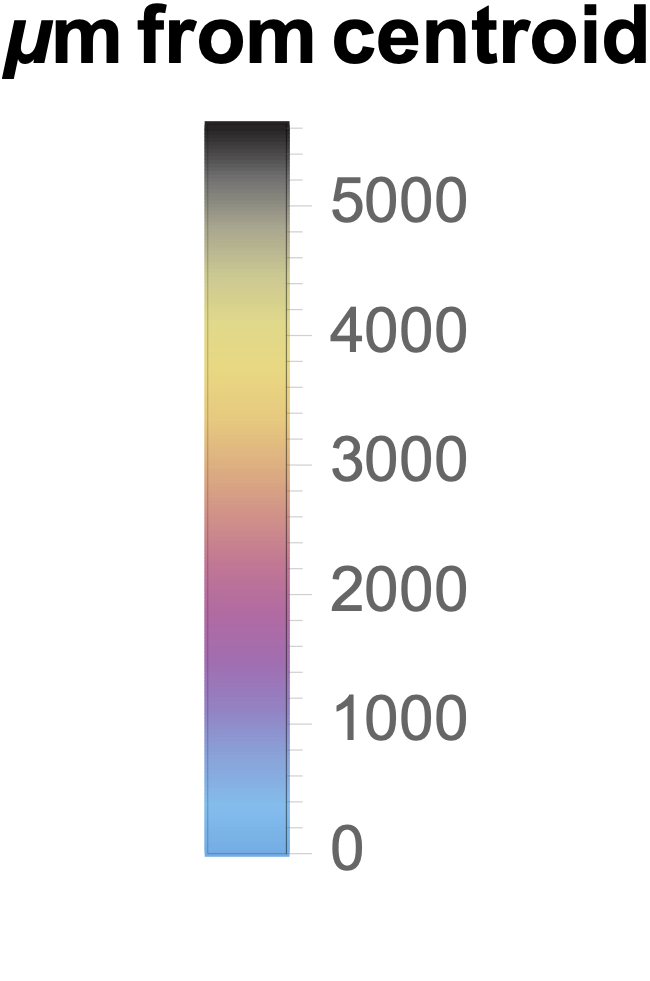}
    
     \caption{Scatter plots of the degree of each node in the reference graph against its associated intensity, $\rho_i$. $\namedtwo$ is on the left and $\fullone$ is on the right. Colors indicate the distance of that vertex to the centroid of all node. There is a linear relationship between the observed degree and the estimated intensity near the centroid of the neurons.}

    \label{fig:rhodegscatter}
\end{figure}

\section{Simulation results\label{sec:simulation}}

In this section, we provide validation for the model that we have developed in Section \ref{sec:estimatedmodel} and fit in Section \ref{sec:modelparams}. We do this by generating 200 simulated graphs from the connection function (Theorem \ref{res:edge|dist,valence}) and plot the histograms of several key network properties on this sample of 200 graphs. We then show that the reference graphs $\namedtwo$ and $\fullone$ fall well within the range of values from these simulated graphs.

For our simulations we use the estimated parameters in Section \ref{sec:modelparams} and construct 200 graphs for $\namedtwo$ and 100 graphs for $\fullone$. We emphasize that, the positions of the nodes are fixed; however, the assignment of the $\rho_i$ are randomly permuted among them.


\subsection{Edges, vertices and connected components\label{sec:edgesetc.}}

If the connection function produces a good model for the graph, then the simulated graphs would have similar basic graph characteristics, such as the number of edges and number of connected components, etc. In Table \ref{tab:namedstats}, we report the mean and standard deviation of various graph characteristics for models built from $\namedtwo$ and $\fullone$ with the corresponding values from the original graph. For many of the characteristics including the maximum degree, the number of triangles and the number of closed 4-walks, the values for the reference graph are within two standard deviations of the mean of the simulated graphs. Figure \ref{fig:namedhists} provides a comparison of the simulated distribution of these characteristics. In all of these cases, the reference graph has characteristic counts that fall within the range of the simulations. We emphasize that the random assignment of the $\rho_i$ among the vertices appears induce a large variance for each of these characteristics; for instance, if nodes closer to the boundary are assigned larger $\rho_i$, then the resulting network will have fewer edges (and hence triangles and 4-walks, etc.). This variance is realized in the simulations for both $\namedtwo$ and $\fullone$, as many of the parameters for the reference graph fall in the lower of range of the simulations. None the less, the reference graphs are is consistent with the positive correlation among these parameters as illustrated in Figure \ref{sub:listmaxvstri} and \ref{sub:listmaxvstrifull}.

While the simulations produce appropriate numbers of triangles and closed 4-walks as well as a comparable maximum degree, there are significant discrepancies between the model and the reference graph are in the number of  self-loops and the number of connected components. Both of these short-comings can be attributed to the effect of the model on node(s) that are a short distance apart (i.e. the mismatch of the expected and observed values for $F_1(x)$ and $F_2(x)$ near $x = 0$ discussed in Section \ref{sec:modelparams}.)

In particular, from Table \ref{table:limiting values} and Figure \ref{fig:f1andf2fits}, at $x=0$, $F_2$ is significantly more overestimated than $F_1$. As a result, the ratio $\hat F_1/\hat F_2$ underestimates the conditional probability $P(i \sim j | d_{i,j})$. 
However, the connection function given in Theorem \ref{res:edge|dist,valence} depends on the ratio $\hat F_1/\hat F_2$. 
From  Figure \ref{fig:polycompare}, we see that the estimated ratios in that define the connection function are also lower than the empirically observed ratios. 
This leads to an underestimation of the number of self loops and connections at short length in the model than is observed. In particular, if one only considered Remark \ref{rmk:F_1(0)behavior}, were $F_1(0)$ gives the number of self loops, one would expect the simulated graphs to have, on average an larger number of self-loops than the reference graphs, which is in contradiction with the data presented in Table \ref{tab:namedstats}. However, the observed difference in the estimated ratio $\hat F_1/\hat F_2$ versus the observed value of $P(i \sim j | d_{i,j} = 0)$ counteracts this overestimation.

In addition, the biophyisical nature of the reference graph makes it nearly a priori connected with few isolated neurons. For instance, neurons that fail to make appropriate connection often die (see for example \cite{carlson2018human}); such neurons would not be realized in our reference graph. As a result, the discrepancy concerning the number of self-loops and  the number connected components appears can be attributed to the nuances of the specific application of the model (i.e., connectome networks) as opposed to the overall model itself.  




\begin{table}[h!]
\begin{tabular}{|c |c | c | c|| c |c | c|} \hline
 & \multicolumn{3}{c||}{$\namedtwo$} & \multicolumn{3}{c|}{$\fullone$} \\ \cline{2-7}
 & Ref. Graph &mean &St. Dev. &  Ref. Graph &Mean &St.Dev \\ \hline
Self-loops & 71 & 47.265 & 4.60552  & 105 & 69.76 & 7.16913 \\ \hline  
Connected Components& 8 & 20.415 & 3.06098 & 26 & 234.81 & 10.4134\\ \hline  
Non-Iso. Conn. Comp. & 1 & 1.17 & 0.40238 & 7 & 3.46 & 1.77195 \\ \hline
Max Valency & 71 & 100.725 & 13.0237 & 150 & 263.72 & 53.0845 \\ \hline
Number of Edges & 2851 & 2979.1 & 94.9285 & 8089 & 8194.41 & 221.677 \\ \hline
Number of Triangles & 4400 & 5546.18 & 795.223  & 10224 & 9621.73 & 1402.54 \\ \hline
Number of Closed 4-Walks & 981091 & 1.37876 $\times 10^6$ & 236890 & 2.95063 $\times 10^6$ & 3.00237 $\times 10^6$ & 549348. \\ \hline
\end{tabular}
\caption{Graph statistics for 200 graph generated by the model built from the Named graph, as well the values for the Named graph. Many of values for the reference graph are within two standard deviations of the means for the simulated graphs. } 
\label{tab:namedstats}
\end{table}

\begin{figure} [h!] 
  \centering
  \begin{subfigure}[t]{0.4\linewidth}
    \includegraphics[width=\linewidth]{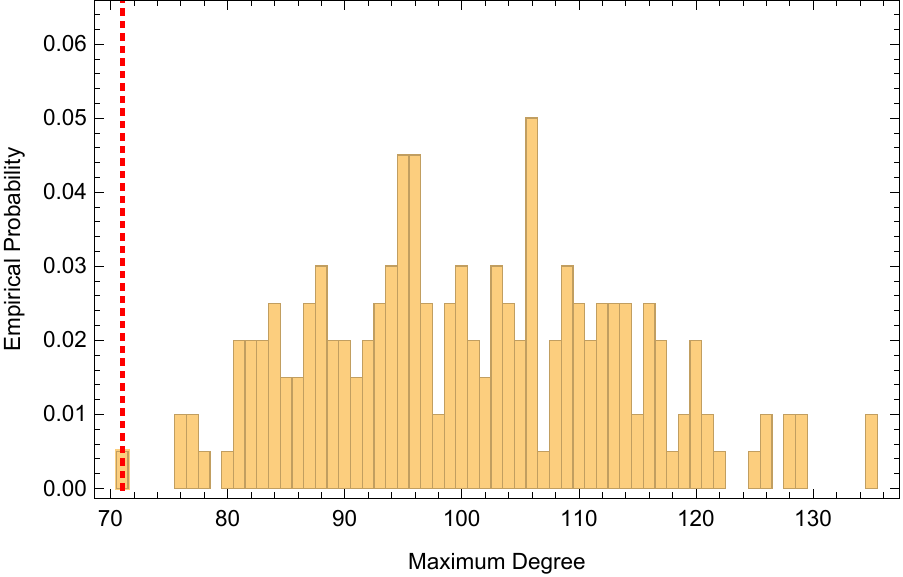}
    \caption{Histogram of maximum degree over 200 model-generated graphs for $\namedtwo$.}
    \end{subfigure} \hfill
    \begin{subfigure}[t]{0.4\linewidth}
    \includegraphics[width=\linewidth]{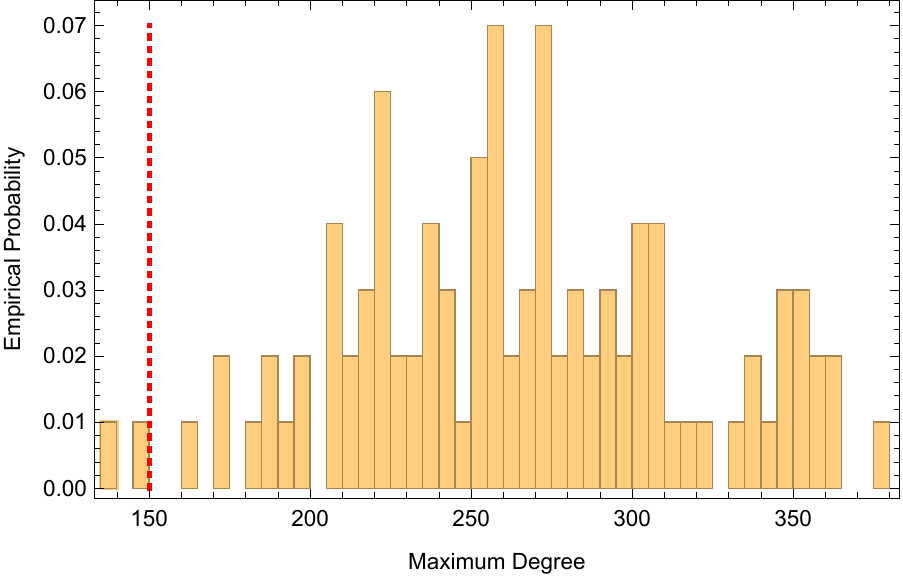}

    \caption{Histogram of maximum degree over 100 model generated $\fullone$.}
    \end{subfigure} \hfill
    \begin{subfigure}[t]{0.4\linewidth}
    \includegraphics[width=\linewidth]{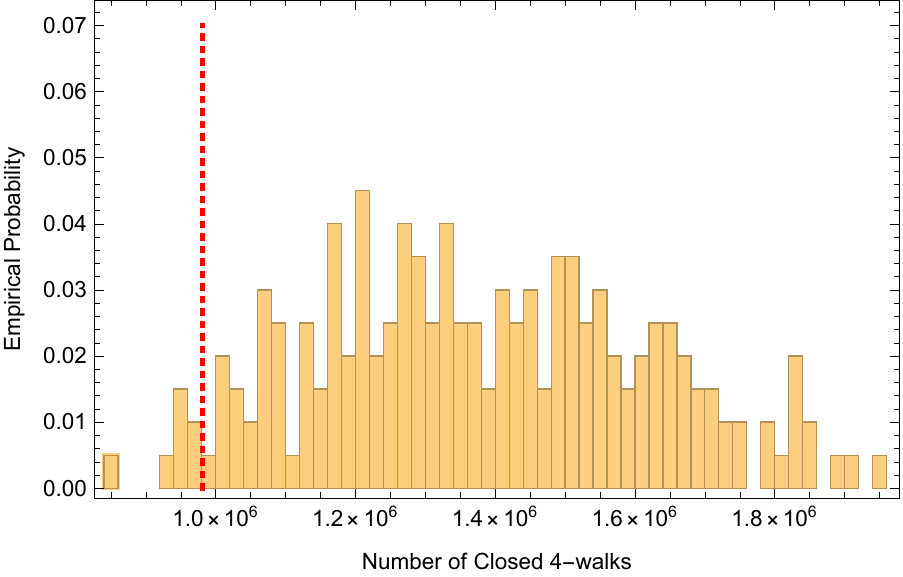}
    \caption{Histogram of the number of closed 4-walks over 200 model-generated graphs for $\namedtwo$.} 
    \end{subfigure} \hfill
     \begin{subfigure}[t]{0.4\linewidth}
    \includegraphics[width=\linewidth]{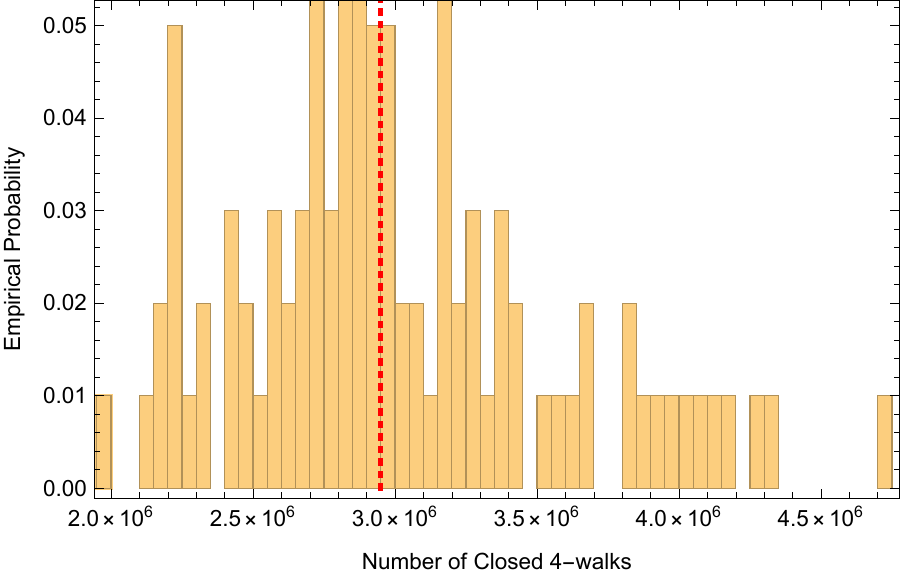}
    \caption{Histogram of the number of closed 4-walks over 100 model-generated graphs for $\fullone$.} 
    \end{subfigure} \hfill
    \begin{subfigure}[t]{0.4\linewidth}
    \includegraphics[width=\linewidth]{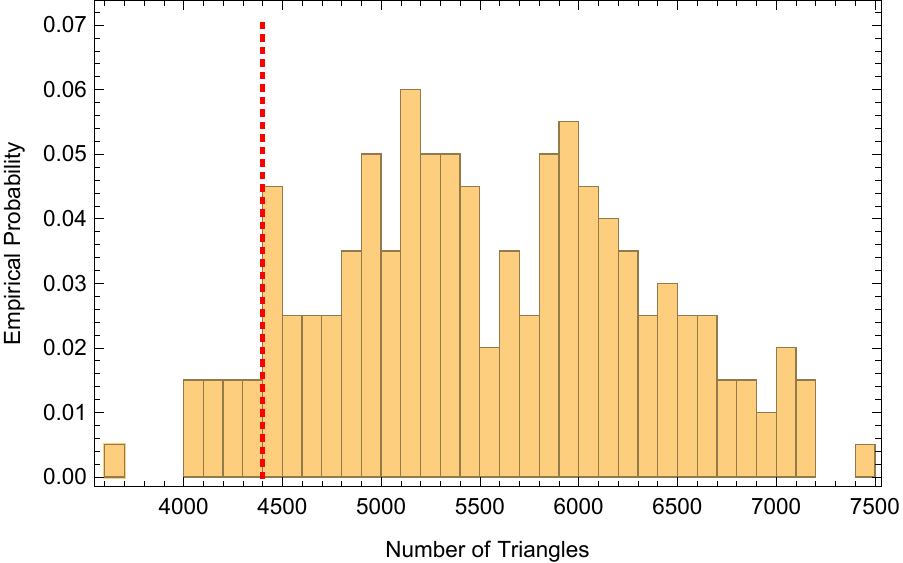}
    \caption{Histogram of three cycles over 200 model-generated graphs for $\namedtwo$.}
    \end{subfigure} \hfill
     \begin{subfigure}[t]{0.4\linewidth}
    \includegraphics[width=\linewidth]{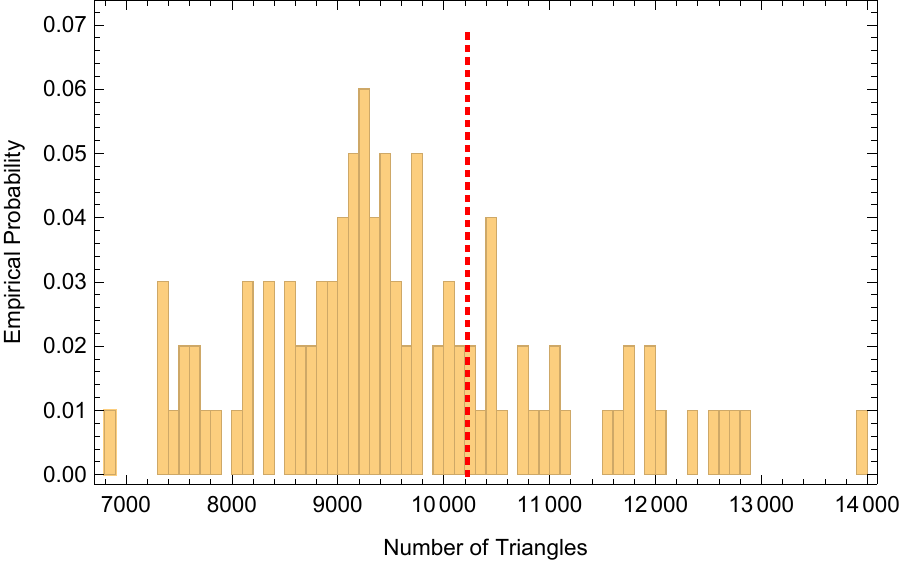}
    \caption{Histogram of three cycles over 100 model-generated graphs for $\fullone$.}
    \end{subfigure} \hfill
    \begin{subfigure}[t]{0.4\linewidth}
    \includegraphics[width=\linewidth]{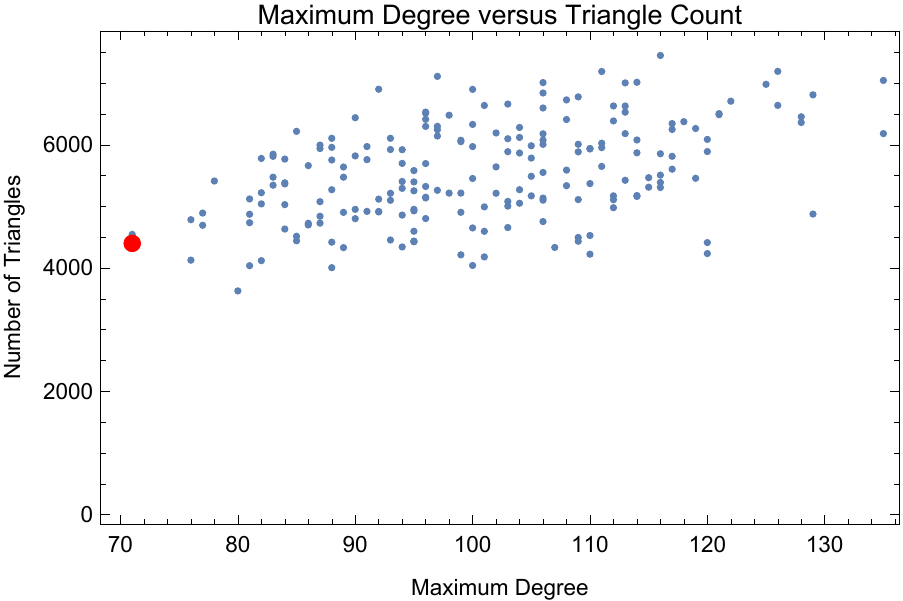}
    \caption{Point plot of the maximum degree versus the number of triangles over 200 model-generated graphs for $\namedtwo$.} 
    \label{sub:listmaxvstri}
    \end{subfigure} \hfill
    \begin{subfigure}[t]{0.4\linewidth}
    \includegraphics[width=\linewidth]{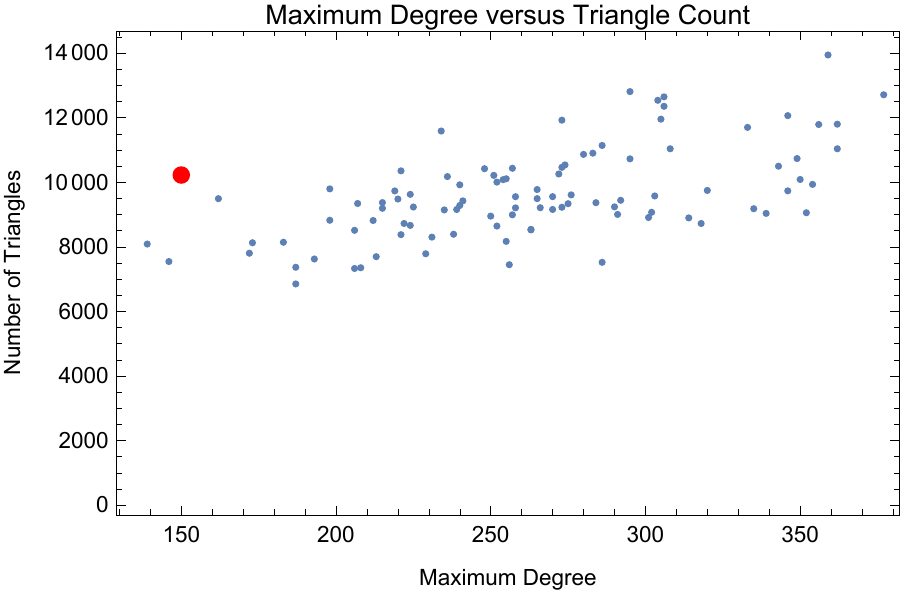}
    \caption{Point plot of the maximum degree versus the number of triangles over 100 model-generated graphs for $\fullone$.}
   \label{sub:listmaxvstrifull}.
    \end{subfigure} \hfill
    
  \caption{Histograms for various graph parameters of randomly-generated graphs from the model built from based on $\namedtwo$ (left) and $\fullone$ (right) with the value for the reference graph indicated in red. In all these cases, the reference graph observation lies within the support of the simulation observations.}
\label{fig:namedhists}
\end{figure}



\subsection{Eigenvalues}

One method for comparing networks is to compare the eigenvalues of various graph-theoretic matrices of the two networks, such as the adjacency matrix or Laplacian matrix. The eigenvalues of these matrices strongly correspond to the local and global properties of the network such as degree, clustering coefficient, motif counts and so on \cite{chung1997spectral}. It possible for two non-isomorphic networks to have the exact same spectrum (i.e., set of eigenvalues) for these graph-theoretic matrices \cite{godsil1982constructing}; however, such cospectral networks have similar properties. Hence, a network model used to generate networks similar to a desired target network, such as in our case, should produce networks with a similar spectrum.

Figure \ref{fig:eigenvalues} shows the histograms for the eigenvalues of the adjacency matrices for both $\namedtwo$ and $\fullone$, overlaid with eigenvalues for a corresponding simulated graph. In general, the randomly generated network has a similar distribution of eigenvalues to the appropriate reference graph. One noticeable discrepancy is that the randomly generated network has distinctly more eigenvalues near 0. In fact, the multiplicity of the exact eigenvalue of 0 for the randomly generated network is significantly higher that the fixed network. We attribute this phenomena to the fact that the randomly generated network is likely to create isolated nodes with no incident edges whereas the original network is {\it a priori} taken to be connected. Each isolated node corresponds exactly to an additional instance of 0 as an eigenvalue of the adjacency matrix. Hence, since our model produces more isolated nodes than the reference graph, one should expect substantially higher multiplicity of 0 as an eigenvalue in the simulations.

\begin{figure} [h!]
  \centering
      \begin{subfigure}[t]{0.4\linewidth}
    \includegraphics[width=\linewidth]{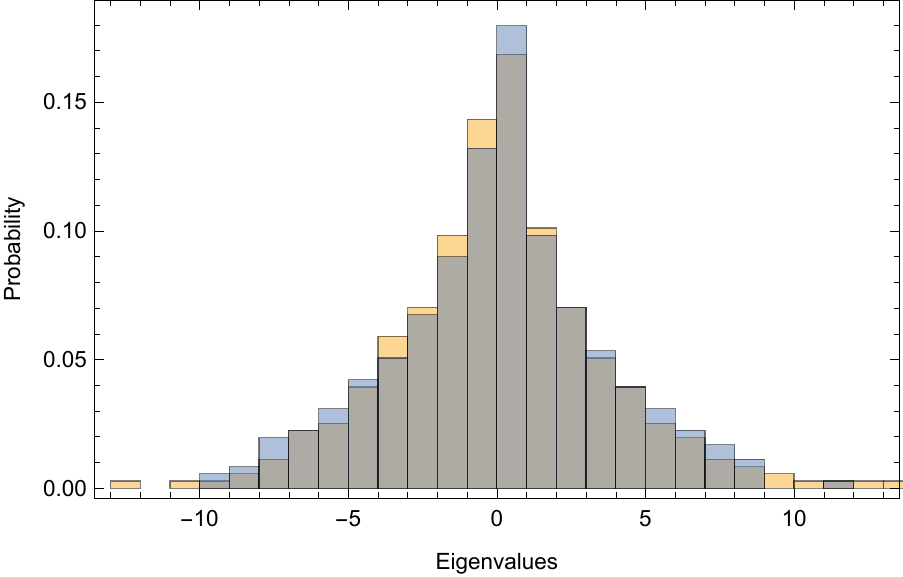}
    \caption{A histogram of the eigenvalue densities for the adjacency matrix for the reference graph $\namedtwo$ (yellow) compared that of a randomly-generated model graph (blue).}
    \end{subfigure} \hfill
    \begin{subfigure}[t]{0.4\linewidth}
    \includegraphics[width=\linewidth]{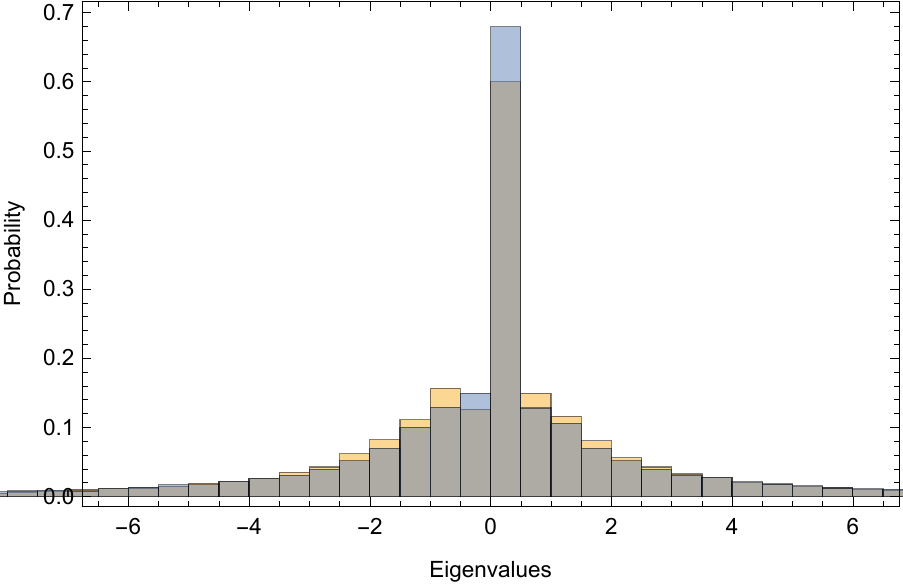}
    \caption{A histogram of the eigenvalue densities for the adjacency matrix for the reference graph $\fullone$ (yellow) compared that of a randomly-generated model graph (blue).} 
    \end{subfigure} \hfill
  \caption{The distribution of eigenvalues for the simulated graphs is close to the distribution of eigenvalues for the reference graphs.}
  \label{fig:eigenvalues}
\end{figure}

\subsection{Centrality Measures}

Next, we compare the centrality measure of the graphs generated by the two models and to those of the original graphs. {\it Centrality measures} are various heuristic measures to determine how important a vertex is based upon the structure of the network. There are many different centrality measures; here, we focus on three traditional measures: betweenness, closeness, and eigenvector centrality.  The {\it betweenness} of a vertex $i$ is $\sum_{u,v} f_{u,v}(i)$ where $f_{u,v}(i)$ is the proportion of shortest paths between $u$ and $v$ that contain $i$. The {\it closeness} of a vertex $i$ is $N / \sum_{j \ne i} d_{i,j}$ where $1/\infty$ is taken to be 0. And the {\it eigenvector centrality} of vertex $i$ is given by the principle eigenvector of the adjacency matrix of the network corresponding to the maximum eigenvalue; Perron-Frobenius theory guarantees such an eigenvalue exists and its eigenvector is nonnegative. 

We focus on these centrality measures, as they generally capture distinct aspects of the graph including local motifs (eigenvector centrality), distance (closeness), and flow (betweenness). As such, these measures are typically hard to replicate jointly. Figure \ref{fig:centrality} shows the plots for eigenvector centrality, betweenness, and closeness. For $\namedtwo$,  the simulated network emulate the centralities of the reference graph very well. On the other hand, for $\fullone$, the centralities of the simulated network do not match as well. We attribute this to the fact that the network data originates from focusing on select nodes (neurons), and that $\fullone$ is mostly-star like. On the other hand, the model we present here does not indicate any focused nodes. This discrepancy effects the local structure of the simulated graph which in turn, affects the centralities measures of the vertices. In contrast, the named graph, is effectively an induced subgraph on the selected nodes.

\begin{figure}[h!]
  \centering
    \begin{subfigure}[t]{0.4\linewidth}
    \includegraphics[width=\linewidth]{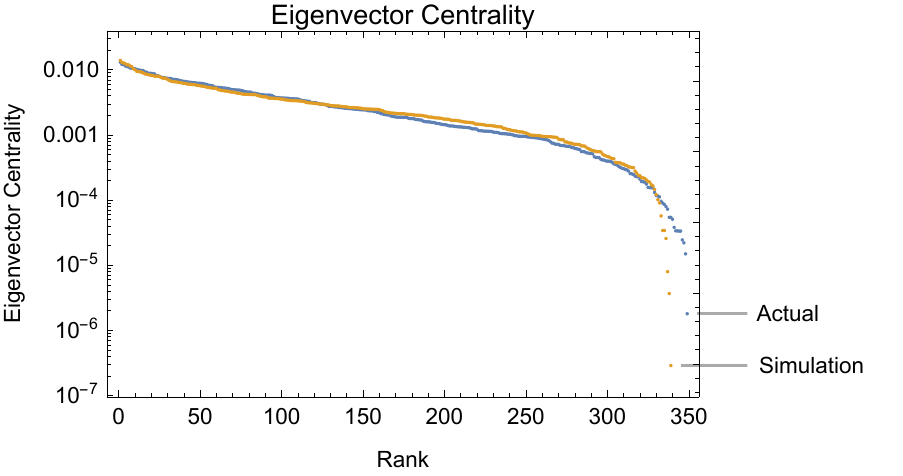}
    \caption{Eigenvector centrality of the Named graph compared to a randomly generated graph generated by the corresponding model.} 
    \end{subfigure} \hfill
    \begin{subfigure}[t]{0.4\linewidth}
    \includegraphics[width=\linewidth]{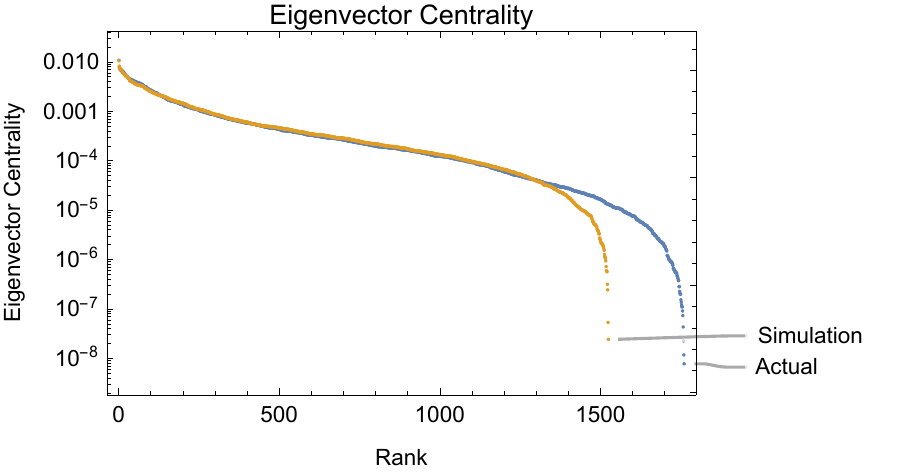}
    \caption{Eigenvector centrality of the Full connectome compared to a randomly generated graph generated by the corresponding model.}
    \end{subfigure} \hfill
    \begin{subfigure}[t]{0.4\linewidth}
    \includegraphics[width=\linewidth]{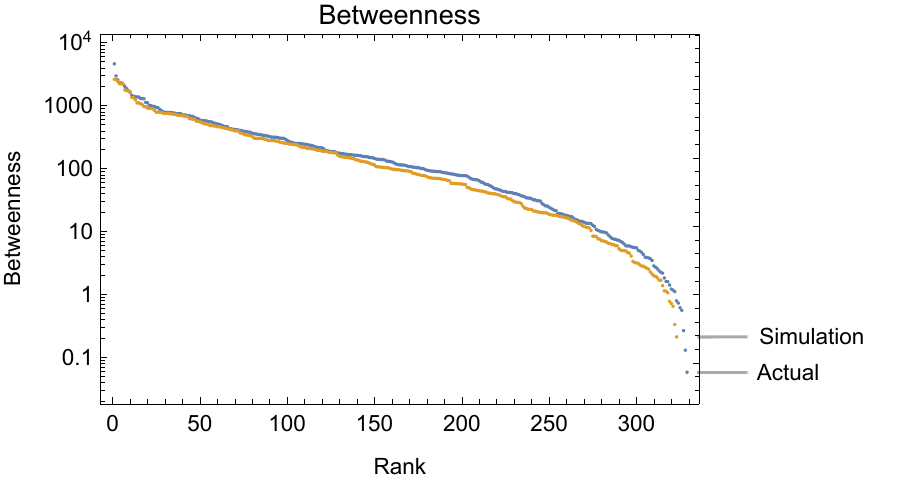}
    \caption{Betweenness measure of vertices in the Named graph compared to a randomly generated graph generated by the corresponding model.} 
    \end{subfigure} \hfill
    \begin{subfigure}[t]{0.4\linewidth}
    \includegraphics[width=\linewidth]{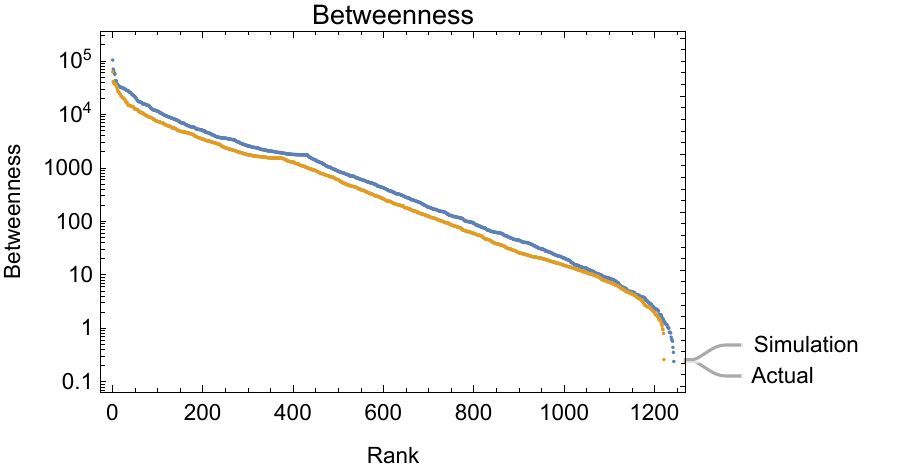}
    \caption{Betweenness measure of vertices in the Full connectome compared to a randomly generated graph generated by the corresponding model.}
    \end{subfigure} \hfill
    \begin{subfigure}[t]{0.4\linewidth}
    \includegraphics[width=\linewidth]{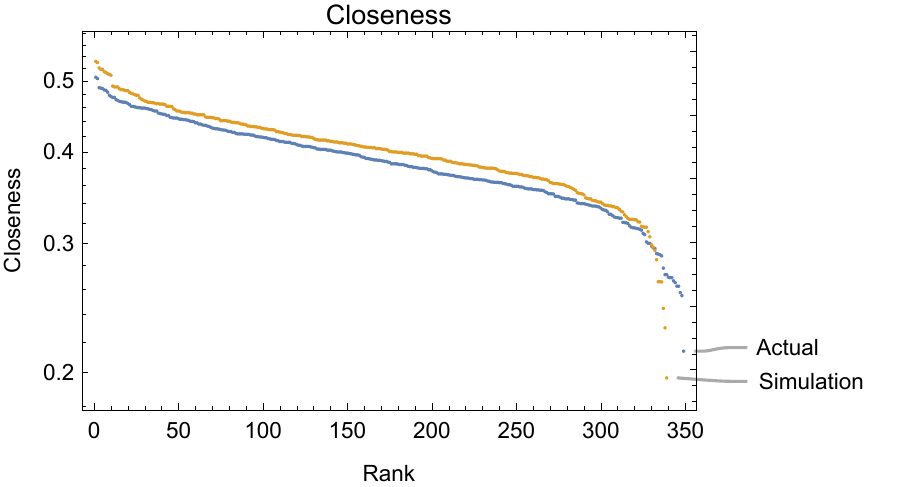}
    \caption{Closeness of vertices in the Named graph compared to a randomly generated graph generated by the corresponding model.} 
    \end{subfigure} \hfill
    \begin{subfigure}[t]{0.4\linewidth}
    \includegraphics[width=\linewidth]{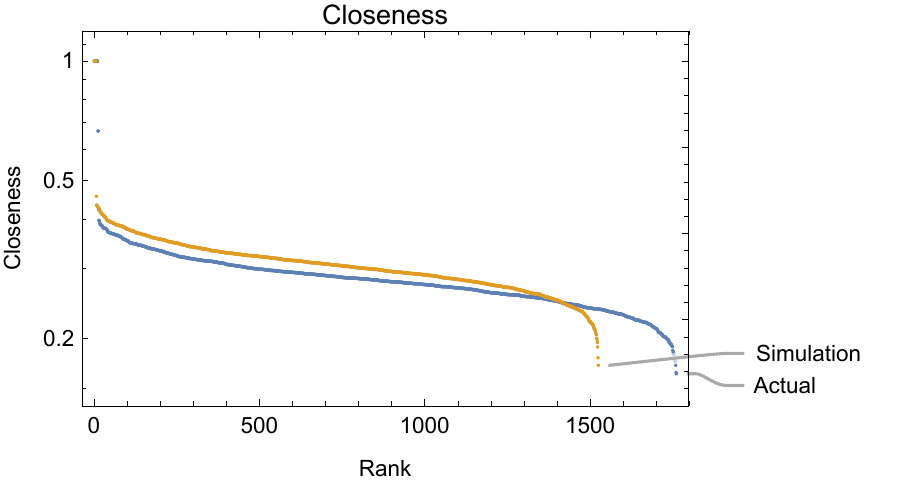}
    \caption{Closeness of vertices in the Full connectome compared to a randomly generated graph generated by the corresponding model.}
    \end{subfigure} \hfill
  \caption{Centrality comparisons of a representative sample from the model built from the Named graph (left) or the Full graph (right) and the reference graph. In all these cases, the centrality measures, sorted by rank, mimic that of the reference graph.}
  \label{fig:centrality}
\end{figure}

\section{Comparison to other models \label{sec:othermodels}}

The model built in this paper is built to accommodate the biological structure of neurons, which are more likely to form connections with near by neurons than ones further away \cite{connectometutorial, spatialfor, spatialagainst}. While much of this work has been conducted on a meso scale, in this paper  we have shown a correlation between the distance between two nodes and the presence of an edge between them (Figures \ref{fig:df1df2fits} and \ref{fig:polycompare}). Even so, it is plausible that a simpler network model with fewer parameters generates the network which is later embedded spatially in a more optimized way (e.g., like a spring network).

The logistic fit for $\hat F_1$ and $\hat F_2$ in this generic geometric Chung-Lu model is much more complicated than existing models such as the classical Chung-Lu model and the polynomial decay models such as GIRG model (as mentioned in Section \ref{sec:litreview}).
However when we compare our logistically fit generic geometric model to both the Chung-Lu model and the polynomial decay models, we see that the our model does indeed have a better fit.



 The classical Chung-Lu model has no consideration for the physical geometry of the vertices. Hence, the nature of clustering should be distinct in the classical model over geometric ones. The Chung-Lu model generates a network whose nodes are substantially closer together than the logistic-geometric model, resulting in a substantially higher closeness centrality. Figure \ref{fig:closenesscomparison} illustrates that the closeness centrality is nearly-uniformly higher in the classical Chung-Lu model. This stands in direct contrast to the closeness of the logistic-geometric model (see Figure \ref{fig:centrality}) which far better mimics the reference graphs. 

Figure \ref{fig:trianglecomparison} compares the number of triangles each vertex is involved in for a single sample graph from each model compared to the reference graph, and to a random graph generated by a simple Chung-Lu model. 

\begin{figure} [h!]
  \centering
    \begin{subfigure}[t]{0.4\linewidth}
    \includegraphics[width=\linewidth]{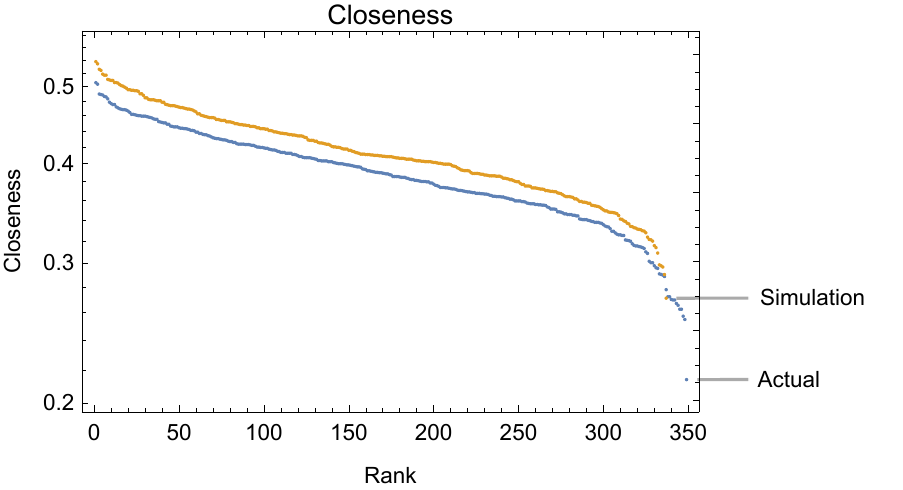}
    \caption{The closeness centrality of a random graph based on the classical Chung-Lu network using $\namedtwo$ as reference.} 
    \end{subfigure} \hfill
    \begin{subfigure}[t]{0.4\linewidth}
    \includegraphics[width=\linewidth]{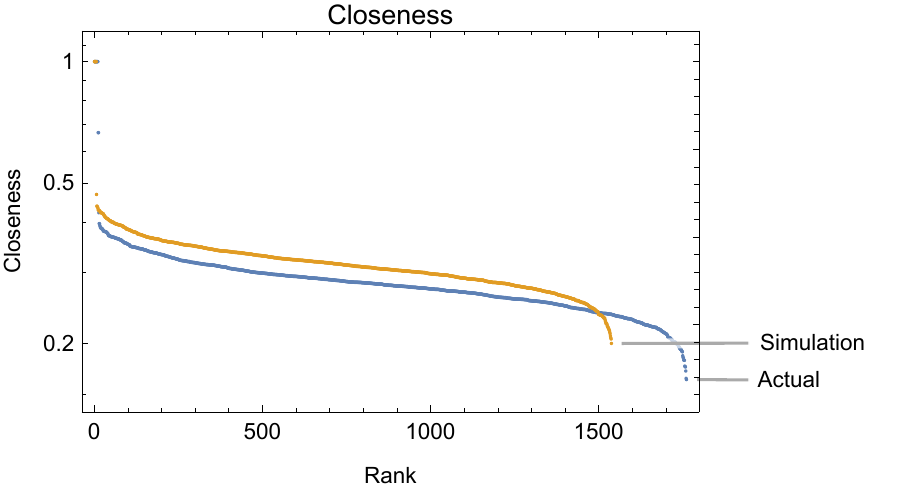}
    \caption{The closeness centrality of a random graph based on the classical Chung-Lu network using $\fullone$ as reference.} 
    \end{subfigure} \hfill
  \caption{Closeness centrality of the classical Chung-Lu model for the reference graphs $\namedtwo$ and $\fullone$. The closeness of the Chung-Lu model is significantly higher than the reference graph compared to the generalized geometric model as seen in \ref{fig:centrality}.}
    \label{fig:closenesscomparison}
\end{figure}

\begin{figure} [h!]
  \centering
    \begin{subfigure}[t]{0.4\linewidth}
    \includegraphics[width=\linewidth]{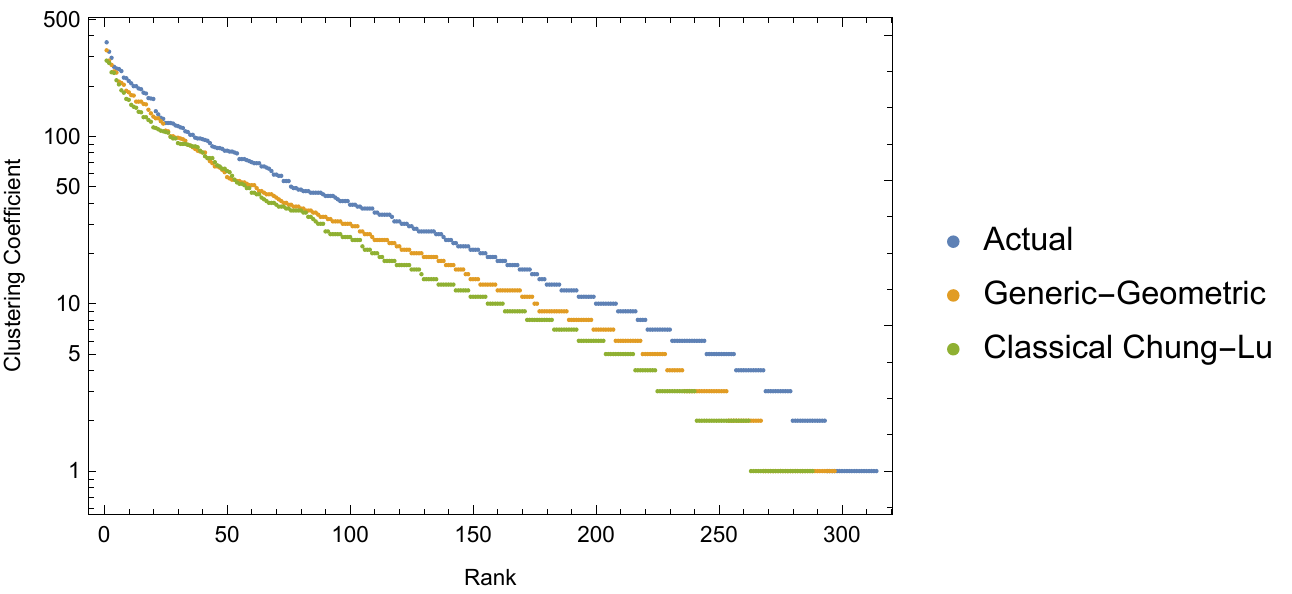}
    \caption{Number of triangles each vertex is involved in for the model generated by the Named graph, a Chung-Lu model of the Named graph, and the Named graph. } 
    \end{subfigure} \hfill
    \begin{subfigure}[t]{0.4\linewidth}
    \includegraphics[width=\linewidth]{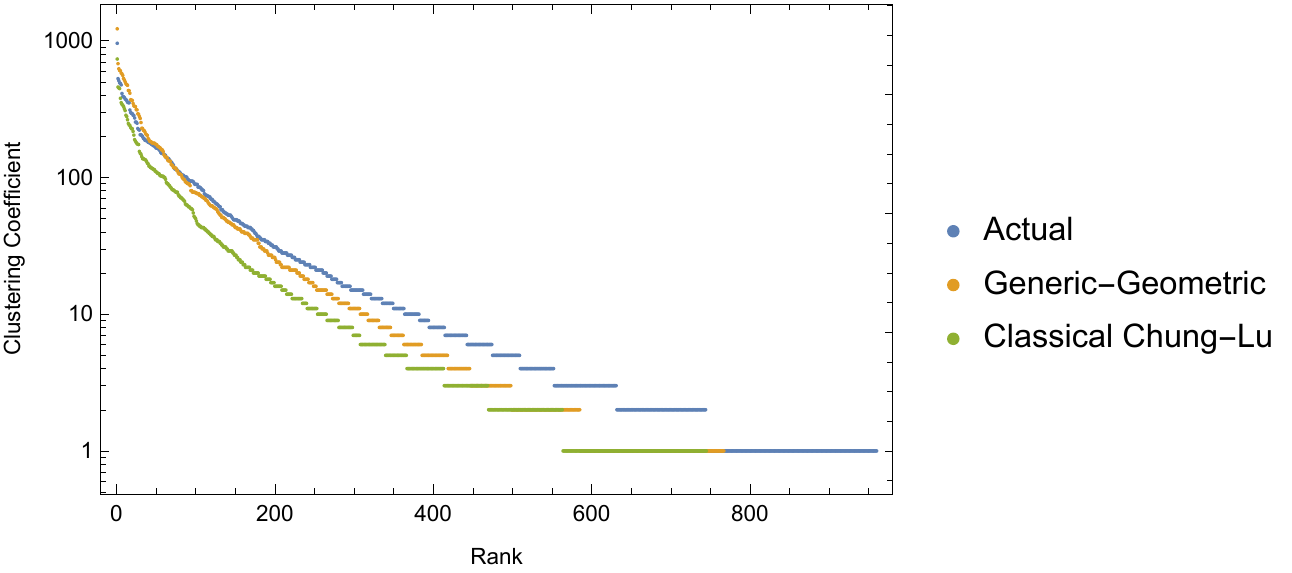}
    \caption{Number of triangles each vertex is involved in for the model generated by the Full connectome, a Chung-Lu model of the Full connectome, and the Full connectome.}
    \end{subfigure} \hfill
  \caption{The clustering coefficients (or ``triangle degree''), sorted by rank, of the original connectome, minus two vertices with high degree (blue), a simulated graph based on the generalized geometric (orange) and a simulated one based on the original Chung-Lu Model (green). This is done for the Named graph (a) and the Full connectome (b). Observe the triangle degrees for the generalized geometric model are closer to that for the reference graph than the Chung-Lu model. }
  \label{fig:trianglecomparison}
\end{figure}

Previous geometric generalizations of the Chung-Lu model (i.e., GIRG. etc.) have utilized connection functions with polynomial decay. In Figure \ref{fig:polycompare} illustrates that the best fits for a connection function with polynomial decay $(e.g., k x^{-\beta})$ are unable to capture the appropriate decay rate for longer distance functions.

\begin{figure} [h!]
\centering
\begin{subfigure}[t]{0.4\linewidth}
\includegraphics[width=\linewidth]{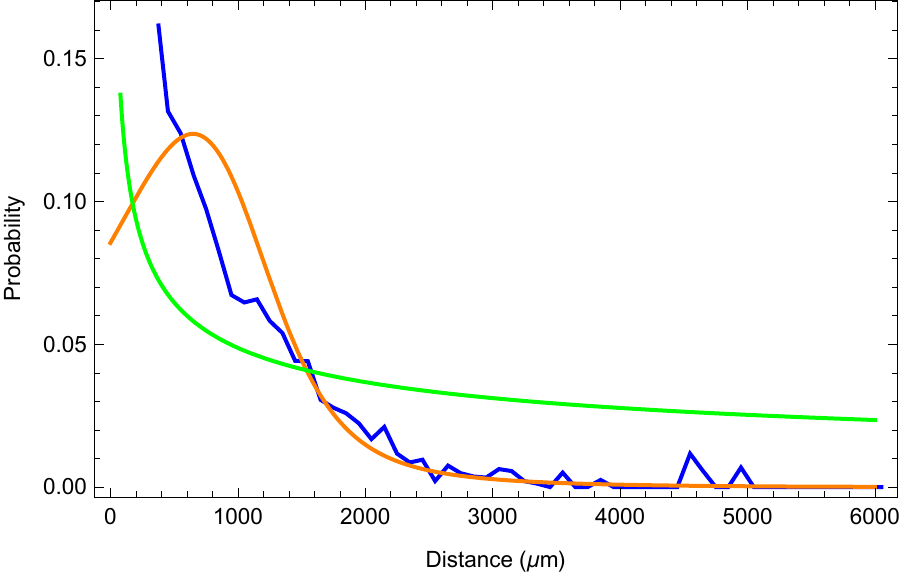} 
\caption{Plot of the empirical probability $y=P(i\sim j | d_{i,j} = x)$ (blue), the modelled function $F_1'(x)/F_2'(x)$ (orange) and the best inverse power fit $y= k x^{-\beta} = 18.0418 x^{-0.8088}$ (green) for $\namedtwo$.}
\end{subfigure} \hfill
\begin{subfigure}[t]{0.4\linewidth}
\includegraphics[width=\linewidth]{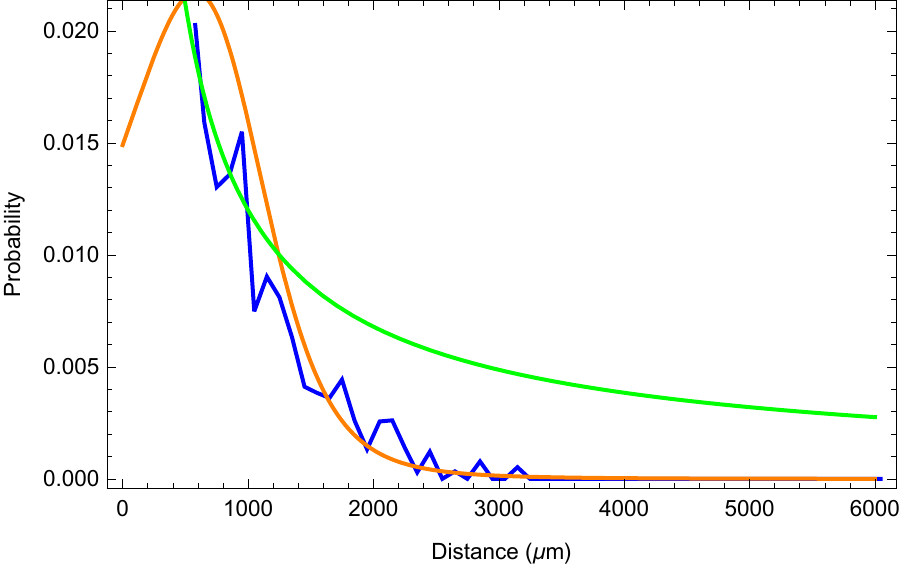}
\caption{Plot of the empirical probability $y=P(i\sim j | d_{i,j} = x)$ (blue), the modelled function $F_1'(x)/F_2'(x)$ (orange) and the best inverse power fit $y= k x^{-\beta} = 1.40837 x^{-0.732098}$ (green) for $\fullone$.}
\end{subfigure} \hfill
\caption{The empirical probabilities of edges given distance (blue) versus modeled probabilities using both the logistic model (orange) and an inverse power model (green). Observe that the inverse power model is unable to effectively model the edge-distance probabilities.}
\label{fig:polycompare}
\end{figure}

Finally, we end with a note on the nonmonotonicity of the connection function developed in this paper (see Figures \ref{fig:df1df2fits} and Figure \ref{fig:polycompare}). A nonmonotonic connection function may appear to be unintuitive. However, Figure \ref{fig:df1df2fits} shows that both graphs $\fullone$ and $\namedtwo$ have a non-monotonic $P(i\sim | d_{i,j}, \; \hat\rho_i, \; \hat\rho_j)$.  Specifically, Tables \ref{tab:nonmonotone} and  \ref{tab:nonmonotone2} shows that for both $\namedtwo$ and $\fullone$ there is an increase in the probabilities for connections in the short range while this probability decreases as the distances get large. For another example, probabilities in endemic networks have been observed to be non-monotonic and logistic-like \cite{lang2018analytic}. 

Mathematically, the ratio of the derivative of two logistic functions (e.g., $\frac{F_1'(x)}{F_2'(x)}$) has the form
 \[\left. K~ \frac{ exp(\alpha_1+\beta_1 x) }{ (1+exp(\alpha_1+\beta_1 x))^2 } \right/ \frac{ exp(\alpha_2+\beta_2 x) }{ (1+exp(\alpha_2+\beta_2 x))^2 } \]
which will be nonmonotonic and have a local maximum whenever $\beta_1 < \beta_2 < 0$ as is the case here. Upon considering this, we should expect the connection function to be nonmonotonic for small distances.

\begin{table}
\centering
\begin{tabular}{|c|c|}
\hline
 $P(i \; \text{self-loop})$     & .199  \\ \hline
 $P(i\sim j \;|\: 0<d_{ij}<200) $ & .270 \\ \hline
\end{tabular}
\caption{Empirical probabilities for connections of short length in $\namedtwo$. Observe the non-monotonic nature of the empirical probabilities.}
\label{tab:nonmonotone}
\end{table}

\begin{table}
\centering
\begin{tabular}{|c|c|}
\hline
 $P(i \; \text{self-loop})$     & .0.0727\\ \hline
 $P(i\sim j \;|\: 0<d_{ij}<100) $ & 0.107 \\ \hline
  $P(i\sim j \;|\: 100<d_{ij}<200) $ & 0.0435 \\ \hline
\end{tabular}
\caption{Empirical probabilities for connections of short length in $\fullone$. Observe the non-monotonic nature of the empirical probabilities.}
\label{tab:nonmonotone2}
\end{table}

\section{Conclusion}

In this paper, we have developed a hybrid model for spatial networks with heterogeneous nodes, which we call the generic geometric Chung-Lu model. The connection function for this model is given in Theorem \ref{res:edge|dist,valence}. This model is designed specifically to model real world situation, where key  assumptions from the graph theory literature, namely, uniform distribution of the nodes on a torus, is relaxed.

After relaxing these constraints, we use this model to generate networks that are similar to the connectome Drosophila Medulla. Connectomes are a natural candidate for this type of geometric model with heterogeneous nodes because of the biological evidence both that different types of neurons have different number of connections to other neurons, and that, due to energy conservation, neurons are more likely to connect to other neurons that are physically closer.  We find that the model developed in this paper does a particularly good job of modeling the features of the graph that control the information flow across the network (i.e. eigenvalues, and centrality measures). 

The approach taken in this paper is particularly well suited for generating graphs with similar information flow properties to a given handful of reference graphs. By relaxing the distributional and periodic constraints on the model, we allow for a paradigm where one can learn the distribution of nodes and the connections function from the reference data. In this manner, the work in this paper can be adapted to a large number of scenarios where there is reason to believe that the network has heterogeneous nodes and has a spatial component to its connection function. For instance, ground based and air transport networks are two other classes of data sets that include distances between the nodes. Both data sets have key features that differentiated it from the assumptions in this paper. (For instance, the there are non-planarity costs in all rail networks, and air networks are both more significantly and much more explicitly non-uniform than the connectome studied here.) However, we believe that, given the assumptions made in this paper, the model developed here can synthesize these real world networks with greater accuracy than currently existing models.

\bibliographystyle{plain}
\bibliography{allbib}

\end{document}